\documentclass[12pt]{amsart}

\usepackage{amssymb, amscd, txfonts}
\usepackage{graphicx}
\usepackage[all]{xy} 

\numberwithin{equation}{section}

\newtheorem{theorem}{Theorem}[section]
\newtheorem{proposition}[theorem]{Proposition}
\newtheorem{lemma}[theorem]{Lemma}
\newtheorem{corollary}[theorem]{Corollary}

\theoremstyle{definition}
\newtheorem{definition}[theorem]{Definition}
\newtheorem{example}[theorem]{Example}

\theoremstyle{remark}
\newtheorem{remark}[theorem]{Remark}

\renewcommand{\hom}{\operatorname{Hom}}
\renewcommand{\ker}{\operatorname{Ker}}

\newcommand{\Z}{\mathbb{Z}}
\newcommand{\Q}{\mathbb{Q}}
\newcommand{\R}{\mathbb{R}}
\newcommand{\C}{\mathbb{C}}

\newcommand{\proj}{{\mathbb P}}
\newcommand{\id}{{\rm id}}
\newcommand{\GL}{{\rm GL}}

\newcommand{\OL}{{\rm O}^{+}(L)}
\newcommand{\Or}{{\rm O}^{+}}
\newcommand{\Ost}{\widetilde{{\rm O}}^{+}(L)}

\newcommand{\DL}{\mathcal{D}_{L}}
\newcommand{\D}{\mathcal{D}}
\newcommand{\DI}{\mathcal{D}(I)}
\newcommand{\FG}{\mathcal{F}(\Gamma)}
\newcommand{\FGcpt}{\mathcal{F}(\Gamma)^{\Sigma}}

\newcommand{\G}{\Gamma}
\newcommand{\GG}{\langle \Gamma, -{\rm id} \rangle}
\newcommand{\GIQ}{\Gamma(I)_{\mathbb{Q}}}
\newcommand{\GIZ}{\Gamma(I)_{\mathbb{Z}}}
\newcommand{\UIQ}{U(I)_{\mathbb{Q}}}
\newcommand{\UIZ}{U(I)_{\mathbb{Z}}}
\newcommand{\UIQZ}{U(I)_{\mathbb{Q}/\mathbb{Z}}}
\newcommand{\UIZZ}{U(I)_{\mathbb{Z}}'}
\newcommand{\GIZbar}{\overline{\Gamma(I)}_{\mathbb{Z}}}
\newcommand{\GIQbar}{\overline{\Gamma(I)}_{\mathbb{Q}}}
\newcommand{\CI}{\mathcal{C}_{I}}
\newcommand{\CII}{\mathcal{C}_{I}^{+}}

\newcommand{\GJQ}{\Gamma(J)_{\mathbb{Q}}}
\newcommand{\WJQ}{W(J)_{\mathbb{Q}}}
\newcommand{\VJQ}{V(J)_{\mathbb{Q}}}
\newcommand{\UJQ}{U(J)_{\mathbb{Q}}}
\newcommand{\GJZ}{\Gamma(J)_{\mathbb{Z}}}
\newcommand{\WJZ}{W(J)_{\mathbb{Z}}}
\newcommand{\VJZ}{V(J)_{\mathbb{Z}}}
\newcommand{\UJZ}{U(J)_{\mathbb{Z}}}
\newcommand{\UJZZ}{U(J)_{\mathbb{Z}}'}
\newcommand{\GJZbar}{\overline{\Gamma(J)}_{\mathbb{Z}}}
\newcommand{\GJQbar}{\overline{\Gamma(J)}_{\mathbb{Q}}}
\newcommand{\WJQZ}{W(J)_{\mathbb{Q}/\mathbb{Z}}}
\newcommand{\UJQZ}{U(J)_{\mathbb{Q}/\mathbb{Z}}}

\newcommand{\XI}{\mathcal{X}(I)}
\newcommand{\XII}{\mathcal{X}(I)'}
\newcommand{\XIcpt}{\mathcal{X}(I)^{\Sigma}}
\newcommand{\XIIcpt}{(\mathcal{X}(I)')^{\Sigma}}
\newcommand{\TJ}{\mathcal{T}(J)}
\newcommand{\TJcpt}{\overline{\mathcal{T}(J)}}
\newcommand{\XJ}{\mathcal{X}(J)}
\newcommand{\XJcpt}{\overline{\mathcal{X}(J)}}
\newcommand{\VJ}{\mathcal{V}(J)}
\newcommand{\HJ}{\mathbb{H}_{J}}

\begin{document}

\title[]{Irregular cusps of orthogonal modular varieties}
\author[]{Shouhei Ma}
\thanks{Supported by JSPS KAKENHI 17K14158 and 20H00112.} 
\address{Department~of~Mathematics, Tokyo~Institute~of~Technology, Tokyo 152-8551, Japan}
\email{ma@math.titech.ac.jp}
\maketitle

\begin{abstract}
Irregular cusp of an orthogonal modular variety is a cusp 
where the lattice for Fourier expansion is strictly smaller than the lattice of translation. 
Presence of such a cusp affects 
the study of pluricanonical forms on the modular variety using modular forms. 
We study toroidal compactification over an irregular cusp, 
and clarify there the cusp form criterion for the calculation of Kodaira dimension. 
At the same time, we show that irregular cusps do not arise frequently: 
besides the cases when the group is neat or contains $-1$, 
we prove that the stable orthogonal groups of most (but not all) even lattices 
have no irregular cusp. 
\end{abstract}


\section{Introduction}\label{sec: intro}

Irregular cusp of a modular curve is a cusp where 
the width of translation is strictly smaller than the width for Fourier expansion. 
It does not arise frequently, but does exist. 
At such a cusp, vanishing order of cusp forms has to be considered carefully, 
especially when compared with that of pluricanonical forms (cf.~\cite{DS} \S 3.2 -- \S 3.3). 
In this article we study and classify irregular cusps for orthogonal groups of signature $(2, b)$, 
and clarify the effect of such cusps on the study of Kodaira dimension of orthogonal modular varieties. 

Let $L$ be a lattice of signature $(2, b)$. 
Let ${\D}={\DL}$ be the Hermitian symmetric domain attached to $L$, 
which is defined as either of the two connected components of the space 
\begin{equation*}
\{ {\C}\omega \in {\proj}L_{{\C}} \: | \: (\omega, \omega)=0, \: (\omega, \bar{\omega})>0 \}. 
\end{equation*}
We write ${\OL}$ for the subgroup of the orthogonal group ${\rm O}(L)$ 
that preserves the component ${\D}$. 

The domain ${\D}$ has $0$-dimensional and $1$-dimensional cusps. 
For simplicity of exposition, 
we speak only of $0$-dimensional cusps for the moment: 
in fact, the case of $1$-dimensional cusps can be reduced to 
that of adjacent $0$-dimensional cusps. 
A $0$-dimensional cusp of ${\D}$ corresponds to a rank $1$ primitive isotropic sublattice $I$ of $L$. 
Let ${\UIQ}$ be the unipotent part of the stabilizer of $I$ in ${\rm O}^{+}(L_{{\Q}})$. 
Then ${\UIQ}$ is already abelian: 
it is a ${\Q}$-vector space of dimension $b$ (with a hyperbolic quadratic form). 
Let ${\G}$ be a finite-index subgroup of ${\OL}$. 
The cusp $I$ is called an \textit{irregular cusp} for ${\G}$ if 
${\UIQ}\cap {\G} \ne {\UIQ}\cap {\GG}$. 
As we will explain, 
${\UIZ}={\UIQ}\cap {\G}$ is the lattice for Fourier expansion of ${\G}$-modular forms around $I$, 
while ${\UIZZ}={\UIQ}\cap {\GG}$ is the lattice of translation around $I$ in the ${\G}$-action. 
We give several characterizations of irregularity (Proposition \ref{prop: characterize 0dim irregular}), 
including one suitable for explicit calculation. 

Irregular cusps are rather rare: 
they do not exist 
when $-{\rm id}\in {\G}$ or 
when ${\G}$ is neat 
or when ${\G}\subset {\rm SO}^{+}(L)$ with $b$ odd. 
But they do exist, in infinitely many examples in every dimension (\S \ref{ssec: irregular example}). 
Our particular interest is in the so-called \textit{stable orthogonal groups} ${\Ost}$ of even lattices $L$, 
defined as the kernel of the reduction map ${\OL}\to {\rm O}(L^{\vee}/L)$. 
This is the group that most frequently appear in the moduli problem 
related to orthogonal modular varieties. 
Our calculation concerning ${\Ost}$ can be summarized as follows.  

\begin{proposition}[\S \ref{ssec: stable orthogonal}, \S \ref{ssec: irregular example}]\label{prop: intro stable orthogonal}
The stable orthogonal group ${\Ost}$ of an even lattice $L$ 
has no irregular cusp unless when 
$L^{\vee}/L \simeq {\Z}/8\oplus ({\Z}/2)^{\oplus a}$ or 
$L^{\vee}/L \simeq ({\Z}/4)^{\oplus 2}\oplus ({\Z}/2)^{\oplus a}$ as abelian groups. 
Conversely, if $L=U\oplus \langle -8 \rangle \oplus M$ or 
$L=U\oplus \langle -4 \rangle^{\oplus 2} \oplus M$ with $M^{\vee}/M$ 2-elementary, 
then ${\Ost}$ has an irregular $0$-dimensional cusp. 
\end{proposition}

Consequently, we obtain classification 
for the following examples from moduli spaces (\S \ref{sec: example}): 
\begin{itemize}
\item 
The modular group for $K3$ surfaces of degree $2d$ 
has an irregular cusp exactly when $d=4$. 
\item 
The modular group for irreducible symplectic manifolds of $K3^{[t+1]}$-type 
with polarization of split type and degree $2d$ (\cite{GHS10}) 
has an irregular cusp exactly when $(t, d)=(1, 4), (2, 2), (4, 1)$.  
\item 
The modular group for O'Grady10 manifolds 
with polarization of split type and degree $2d$ (\cite{GHS11}), 
which is larger than ${\Ost}$, 
has an irregular cusp exactly when $d=4$.  
\item Similarly, the modular group for deformation generalized Kummer varieties 
with polarization of split type and degree $2d$ (\cite{Da}) 
has an irregular cusp exactly when $(t, d)=(4, 1)$. 
\item We will also cover the groups considered in \cite{TVA}, \cite{Pe}, \cite{FM}. 
\end{itemize}
 
A subtle issue concerning irregular cusps, 
which is the main object of this article, 
is comparison of vanishing order between cusp forms and pluricanonical forms. 
We take a toroidal compactification ${\FGcpt}$ of the modular variety ${\FG}={\G}\backslash{\D}$, 
which is defined by choosing a finite collection $\Sigma=(\Sigma_{I})$ of suitable fans, 
one for each ${\G}$-equivalence class of rank $1$ primitive isotropic sublattices $I$ of $L$. 
A ray $\sigma$ in $\Sigma_{I}$ corresponds to a boundary divisor $D(\sigma)$ of 
the torus embedding $\overline{{\D}/{\UIZ}}$, 
and thus determines a boundary divisor $\Delta(\sigma)$ of ${\FGcpt}$ as the image of $D(\sigma)$. 
Then the projection $\overline{{\D}/{\UIZ}}\to {\FGcpt}$ is ramified along $D(\sigma)$, 
with index $2$, exactly when $I$ is irregular and the ray $\sigma$ is irregular in the sense of 
Definition \ref{def: irregular ray}. 

Vanishing order $\nu_{\sigma}(F)$ 
of a ${\G}$-modular form $F$ at $D(\sigma)\subset \overline{{\D}/{\UIZ}}$ 
can be measured by Fourier expansion (\S \ref{ssec: vanishing order}): 
this is done with ${\UIZ}$. 
On the other hand, vanishing order of a pluricanonical form $\omega$ on ${\FG}$ 
should be measured at the level of $\Delta(\sigma) \subset {\FGcpt}$: 
this is essentially done with ${\UIZZ}$.  
When $\omega$ is $m$-canonical and corresponds to $F$ 
(of weight $k=mb$ and character $\chi=\det^{m}$), 
we have the relation (Proposition \ref{prop: vanishing order relation}) 
\begin{equation*}
\nu_{\Delta(\sigma)}(\omega) = a_{\sigma} \cdot \nu_{\sigma}(F) - m, 
\end{equation*}
where $a_{\sigma}=1$ if $\sigma$ is regular but 
$a_{\sigma}=1/2$ if $\sigma$ is irregular due to the boundary ramification. 
If we are involved only with modular forms of specific parity of weight $k$, 
namely $k$ even for $\chi=1$ 
or $k \equiv b$ mod $2$ for $\chi=\det$, 
we do not need to worry about irregular cusps 
because we can enlarge ${\G}$ to ${\GG}$ without any loss. 
However, when we use product with a cusp form of weight in the remaining parity 
in the construction of a modular form corresponding to 
a pluricanonical form, 
we cannot add $-{\id}$ to ${\G}$, 
and have to be careful about the coefficient $a_{\sigma}=1/2$ at irregular rays $\sigma$. 

Gritsenko-Hulek-Sankaran \cite{GHS07} gave a criterion, 
called the low weight cusp form trick, 
for ${\FG}$ to be of general type in terms of existence of a certain cusp form. 
It appears that irregular cusps are not covered in \cite{GHS07}, 
essentially by assuming $-{\rm id}\in {\G}$, 
explicitly for $1$-dimensional cusps (\cite{GHS07} p.539) 
and implicitly for $0$-dimensional cusps (see the remark below). 
In view of the coefficient $a_{\sigma}=1/2$ at irregular $\sigma$, 
it seems that the criterion needs to be modified at such boundary divisors. 
The result is summarized as follows 
(compare with \cite{GHS07} Theorem 1.1). 

\begin{theorem}[Theorem \ref{thm: low slope}]\label{thm: intro modified criterion}
Let $L$ be a lattice of signature $(2, b)$ with $b\geq 9$ 
and ${\G}$ be a subgroup of ${\OL}$ of finite index. 
We take a ${\G}$-admissible collection $\Sigma=(\Sigma_{I})$ of fans so that 
$\Sigma_{I}$ is basic with respect to ${\UIQ}\cap {\GG}$ at every $0$-dimensional cusp $I$. 
Assume that there exists a ${\G}$-cusp form $F$ of weight $k<b$ and some character 
satisfying the following: 
\begin{enumerate}
\item $F$ vanishes at the ramification divisor of ${\D}\to {\FG}$. 
\item $\nu_{\sigma}(F)\geq 2$ at every irregular ray $\sigma$ at every irregular $I$. 
\end{enumerate}
Then ${\FG}$ is of general type. 
\end{theorem}

The condition on $\Sigma$ is put to ensure that ${\FGcpt}$ has canonical singularities 
(\cite{GHS07}, \cite{Ma}), and this can always be satisfied. 
When ${\G}$ has no irregular cusp, the condition (2) is vacuous, 
and this is the criterion in \cite{GHS07}; 
the choice of $\Sigma$ does not matter with $F$ and can be dropped (or hidden) from the criterion. 
Even when ${\G}$ has an irregular cusp, 
if the weight $k$ is even for $\chi=1$ or $k\equiv b$ mod $2$ for $\chi=\det$, 
the condition (2) is still automatically satisfied by the cuspidality of $F$ (Proposition \ref{prop: irregular vanishing order}).  
However, when ${\G}$ has an irregular cusp and $k$ belongs to the remaining parity, 
the condition (2) arises, 
and the choice of $\Sigma_{I}$ is then involved with $F$. 
Practically it would not be very easy to check (or achieve) $\nu_{\sigma}(F) \geq 2$ 
for specific $F$ and $\Sigma_{I}$. 
Probably the most plausible scenario would be to expect and check that 
the group ${\G}$ in question has no irregular cusp. 
We could say that this is a small cost for using cusp forms of arbitrary weight. 

By the examples discussed after Proposition \ref{prop: intro stable orthogonal}, 
the general-type results in \cite{GHS07}, \cite{GHS10}, \cite{GHS11}, 
\cite{Da}, \cite{TVA}, \cite{Pe}, \cite{FM} are not affected. 
This is our essential purpose. 

As a final remark concerning irregular cusps, 
it should be remembered that, in \cite{AMRT}, 
subgroups of ${\rm O}^{+}(L_{{\R}})/\pm {\rm id}$ are considered, 
rather than of ${\rm O}^{+}(L_{{\R}})$. 
This means that the given group ${\G}<{\OL}$ is 
replaced by ${\GG}/\pm{\id}$. 
In this situation, it is not ${\UIZ}={\UIQ}\cap {\G}$ but rather ${\UIZZ}={\UIQ}\cap {\GG}$ 
that is written as $U(F)_{{\Z}}$ in the notation of \cite{AMRT}. 
This is a subtle difference that may arise when working with \cite{AMRT}, 
and that is related with irregularity. 

To conclude, 
irregular cusps are cusps where the lattice for Fourier expansion 
is smaller than the lattice of translation. 
It is the central element $-{\rm id}$ in the Lie group ${\rm O}^{+}(L_{{\R}})$ 
that is eventually responsible for the presence of such cusps. 
We need to be careful about such cusps when using a cusp form 
of odd weight with $\chi=1$ or weight $k\not\equiv b$ mod $2$ with $\chi=\det$ 
to construct a pluricanonical form on ${\FGcpt}$, 
as a benefit of not dividing ${\rm O}^{+}(L_{{\R}})$ by $-{\id}$.

I would like to thank 
Valery Gritsenko, Klaus Hulek, Shigeyuki Kondo and Gregory Sankaran 
for their valuable comments.

\setcounter{tocdepth}{1}
\tableofcontents


This article is organized as follows.  
In \S \ref{sec: 0dim cusp} we recall the structure of the stabilizer of a $0$-dimensional cusp. 
In \S \ref{sec: irr 0dim cusp} we define and study irregular $0$-dimensional cusps. 
In \S \ref{sec: example} we give examples of groups ${\G}$ with/without irregular cusp. 
In \S \ref{sec: 1dim cusp} we recall the structure of the stabilizer of a $1$-dimensional cusp. 
In \S \ref{sec: irr 1dim cusp} we study irregular $1$-dimensional cusps. 
In \S \ref{sec: toroidal cpt} we study some basic properties of a toroidal compactification of ${\FG}$. 
In \S \ref{sec: cusp form criterion} we prove Theorem \ref{thm: intro modified criterion}. 
The main contents of this article are contained in 
\S \ref{sec: irr 0dim cusp}, \S \ref{sec: example}, \S \ref{sec: irr 1dim cusp} 
and \S \ref{sec: cusp form criterion}. 
\S \ref{sec: 0dim cusp} and \S \ref{sec: 1dim cusp} are expository, 
but we tried to be rather self-contained because of the subtle nature of irregular cusps 
and for calculation of explicit examples in \S \ref{sec: example}. 
The logical dependence of these sections is as follows. 

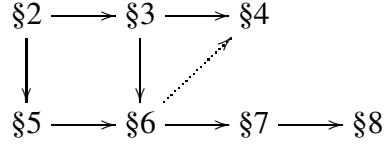
\begin{equation*}
\xymatrix{
\S \ref{sec: 0dim cusp} \ar[r] \ar[d] & \S \ref{sec: irr 0dim cusp} \ar[r] \ar[d] & 
\S \ref{sec: example} &    \\ 
\S \ref{sec: 1dim cusp} \ar[r] & \S \ref{sec: irr 1dim cusp} \ar[r] \ar@{.>}[ur] & 
\S \ref{sec: toroidal cpt} \ar[r]  & \S \ref{sec: cusp form criterion}          
}
\end{equation*}

Throughout the article, a \textit{lattice} usually means a free ${\Z}$-module of finite rank 
endowed with a nondegenerate integral symmetric bilinear form $(\cdot , \cdot) : L \times L \to {\Z}$. 
In a few occasions, we use the word "lattice" just for a free ${\Z}$-module of finite rank, 
but no confusion will likely to occur. 
The dual lattice ${\hom}(L, {\Z})$ of $L$ will be denoted by $L^{\vee}$. 
A sublattice $I\subset L$ is called \textit{primitive} when $L/I$ is free, 
and \textit{isotropic} when $(I, I)\equiv 0$. 
A lattice $L$ is called \textit{even} if $(l, l)\in 2{\Z}$ for every $l\in L$, 
but this is not assumed except in \S \ref{sec: example}.   
We write $U$ for the even unimodular lattice of signature $(1, 1)$ 
given by the Gram matrix 
$\begin{pmatrix} 0 & 1 \\ 1 & 0 \end{pmatrix}$.


\section{$0$-dimensional cusps}\label{sec: 0dim cusp}

Let $L$ be a lattice of signature $(2, b)$. 
We write $Q = Q_{L}$ 
for the isotropic quadric in ${\proj}L_{{\C}}$ defined by $(\omega, \omega)=0$. 
The Hermitian symmetric domain attached to $L$ is the open set of $Q$ 
\begin{equation*}
{\D} = {\DL} = \{ {\C}\omega \in Q | (\omega, \bar{\omega})>0 \}^{+}, 
\end{equation*}
where $+$ means the choice of a connected component. 
%
The domain ${\D}$ has two types of rational boundary components (cusps): 
$0$-dimensional and $1$-dimensional cusps. 
They correspond to primitive isotropic sublattices of $L$ of rank $1$ and $2$ respectively. 
In this section we recall the structure of the stabilizer of a $0$-dimensional cusp 
and partial toroidal compactification over it. 
Although the contents of this section are quite standard (cf.~\cite{Sc}, \cite{Ko}, \cite{GHS07}, \cite{Lo}), 
we tried to be rather self-contained and explicit for two reasons: 
because of the subtle nature of irregular cusps (\S \ref{sec: irr 0dim cusp}), 
and for the sake of calculation of explicit examples (\S \ref{sec: example}).

\subsection{Tube domain model}\label{ssec: tube domain}

Throughout this section we fix a rank $1$ primitive isotropic sublattice $I$ of $L$. 
The $0$-dimensional cusp corresponding to $I$ is the point ${\proj}I_{{\C}}$ of $Q$. 
We write 
\begin{equation*}
L(I) = (I^{\perp}\cap L / I) \otimes I. 
\end{equation*}
Twist by $I$, not choosing its generator, will be rather essential. 
The quadratic form on $I^{\perp}/I$ and an isomorphism $I\simeq {\Z}$ 
define a quadratic form on $L(I)$, 
independent of the choice of $I\simeq {\Z}$, 
which has signature $(1, b-1)$. 
We denote by ${\CI}$ the positive cone in $L(I)_{{\R}}$, 
namely a chosen component of $\{ w \in L(I)_{{\R}} \: | \: (w, w)>0 \}$. 

We write ${\DI}=Q - Q\cap {\proj}I_{{\C}}^{\perp}$. 
Then ${\D}$ is contained in ${\DI}$.  
The projection ${\proj}L_{{\C}}\dashrightarrow {\proj}(L/I)_{{\C}}$ 
from the point ${\proj}I_{{\C}}\in Q$ defines an isomorphism 
\begin{equation*}
{\DI} \: \stackrel{\simeq}{\to} \: {\proj}(L/I)_{{\C}} - {\proj}(I^{\perp}/I)_{{\C}}. 
\end{equation*}
If we choose a rank $1$ sublattice $I'$ of $L$ with $(I, I')\not\equiv 0$, 
this defines a base point of the affine space ${\proj}(L/I)_{{\C}} - {\proj}(I^{\perp}/I)_{{\C}}$ 
and hence an isomorphism 
\begin{equation*}
{\proj}(L/I)_{{\C}} - {\proj}(I^{\perp}/I)_{{\C}} \simeq 
(I^{\perp}/I)_{{\C}}\otimes (I')_{{\C}}^{\vee} \simeq L(I)_{{\C}}. 
\end{equation*}
The image of ${\D}\subset {\DI}$ by this series of isomorphisms is the tube domain in $L(I)_{{\C}}$  
\begin{equation*}
\mathcal{D}_{I} = \{ \: Z\in L(I)_{{\C}} \: | \: {\rm Im}(Z)\in {\CI} \: \}. 
\end{equation*}
This is the tube domain realization of ${\D}$ with respect to $I$. 
Note that this is not canonical, depending on the choice of $I'$ (= choice of base point). 
Change of $I'$ acts by translation on $L(I)_{{\C}}$.

\subsection{Stabilizer over ${\Q}$}\label{ssec: stab 0dim Q}

Let ${\GIQ}$ be the stabilizer of $I$ in ${\rm O}^{+}(L_{{\Q}})$. 
Note that we are not considering the stabilizer of $I_{{\Q}}$, but of $I$. 
This is not restrictive when restricting to subgroups of ${\OL}$. 
We put 
\begin{equation*}
{\UIQ} = {\ker}({\GIQ} \to {\rm O}(I_{{\Q}}^{\perp}/I_{{\Q}})\times {\rm GL}(I)). 
\end{equation*}
This is the unipotent part of ${\GIQ}$ and can be explicitly described as follows. 
For $m\otimes l \in L(I)_{{\Q}}$ the \textit{Eichler transvection} $E_{m\otimes l}\in {\GIQ}$ is defined by 
(cf.~\cite{Sc}, \cite{GHS09}) 
\begin{equation*}
E_{m\otimes l}(v) = v - (\tilde{m}, v)l + (l, v)\tilde{m} - \frac{1}{2}(m, m)(l, v)l, \qquad v\in L_{{\Q}}, 
\end{equation*}
where $\tilde{m}\in I^{\perp}_{{\Q}}=I^{\perp}_{{\Q}}\cap L_{{\Q}}$ 
is an arbitrary lift of $m\in (I^{\perp}/I)_{{\Q}}$. 
In particular, 
$E_{m\otimes l}(v) = v - (m, v)l$ 
when $v\in I^{\perp}_{{\Q}}$. 
We have $E_{w}\circ E_{w'}=E_{w+w'}$ for $w, w'\in L(I)_{{\Q}}$. 
Then we have the canonical isomorphism 
\begin{equation*}
L(I)_{{\Q}} \to {\UIQ}, \qquad m\otimes l \mapsto E_{m\otimes l}. 
\end{equation*}
We identify ${\UIQ}$ with $L(I)_{{\Q}}$ in this way. 
We also identify $\textrm{O}(I_{{\Q}}^{\perp}/I_{{\Q}})\times {\rm GL}(I)$ with 
$\textrm{O}(L(I)_{{\Q}})\times {\rm GL}(I)$ by the canonical twisted isomorphism 
\begin{equation*}
\textrm{O}(I_{{\Q}}^{\perp}/I_{{\Q}})\times {\rm GL}(I) \to \textrm{O}(L(I)_{{\Q}})\times {\rm GL}(I), 
\qquad (\gamma_{1}, \gamma_{2})\mapsto (\gamma_{1}\otimes \gamma_{2}, \gamma_{2}). 
\end{equation*} 

We thus have the canonical exact sequence 
\begin{equation}\label{eqn: canonical sequence GIQ}
0 \to L(I)_{{\Q}}\to {\GIQ} \stackrel{\pi}{\to} {\rm O}^{+}(L(I)_{{\Q}})\times {\rm GL}(I) \to 1. 
\end{equation}
This sequence \textit{non-canonically} splits 
if we choose a lift $I^{\perp}_{{\Q}}/I_{{\Q}}\hookrightarrow I^{\perp}_{{\Q}}$ 
of $I^{\perp}_{{\Q}}/I_{{\Q}}$, or equivalently, 
a rank $1$ sublattice $I'$ of $L$ with $(I, I')\not\equiv 0$. 
In any such a splitting of ${\GIQ}$, 
the element $-{\rm id}_{L}\in {\GIQ}$ is expressed as 
$-{\rm id}_{L} = ({\rm id}_{L(I)}, -{\rm id}_{I}, 0)$ 
where $0\in L(I)_{{\Q}}$. 
Since $\gamma \circ E_{w} \circ \gamma^{-1} =E_{\gamma w}$ for $\gamma \in {\GIQ}$, 
the adjoint action of ${\GIQ}$ on ${\UIQ}$ 
coincides with the natural action of ${\GIQ}$ on $L(I)_{{\Q}}$. 
In the induced adjoint action of 
${\Or}(L(I)_{{\Q}})\times {\rm GL}(I)$ on ${\UIQ}$, 
${\rm GL}(I)=\{ \pm 1 \}$ acts trivially, 
and ${\Or}(L(I)_{{\Q}})$ acts by its natural action on $L(I)_{{\Q}}$.

\subsection{Stabilizer over ${\Z}$}\label{ssec: stab 0dim Z}

Let ${\G}$ be a subgroup of ${\OL}$ of finite index. 
We write 
\begin{equation*}
{\GIZ} = {\GIQ}\cap {\G}, \quad 
{\UIZ} = {\UIQ}\cap {\G}, \quad 
{\GIZbar} = {\GIZ}/{\UIZ}. 
\end{equation*}
Then ${\UIZ}$ is a lattice on ${\UIQ}$. 
By definition we have the exact sequence 
\begin{equation}\label{eqn: GIZ sequence}
0 \to {\UIZ} \to {\GIZ} \to {\GIZbar} \to 1. 
\end{equation}
Although \eqref{eqn: canonical sequence GIQ} splits, 
this does not necessarily mean that \eqref{eqn: GIZ sequence} splits. 
We write 
${\UIQZ}={\UIQ}/{\UIZ}$. 
This is the group of torsion points of the algebraic torus $T(I)=U(I)_{{\C}}/{\UIZ}$. 
We also put 
\begin{equation*}
{\GIQbar} = 
\pi^{-1}({\rm O}^{+}({\UIZ})\times {\rm GL}(I)) / {\UIZ}, 
\end{equation*}
which makes sense because ${\UIZ}$ is normal in 
$\pi^{-1}({\rm O}^{+}({\UIZ})\times {\rm GL}(I))$ 
by definition. 
This fits into the exact sequence 
\begin{equation*}
0 \to {\UIQZ} \to {\GIQbar} \to {\rm O}^{+}({\UIZ})\times {\rm GL}(I) \to 1. 
\end{equation*}
Then ${\GIZbar}$ is canonically a subgroup of ${\GIQbar}$ with ${\GIZbar}\cap {\UIQZ} = \{ 0 \}$. 
Although the projection 
${\GIZbar}\to {\rm O}^{+}({\UIZ})\times {\rm GL}(I)$ 
is injective, we should not view ${\GIZbar}$ as a subgroup of 
${\rm O}^{+}({\UIZ})\times {\rm GL}(I)$ 
when considering its action on ${\D}/{\UIZ}$ (cf.~\cite{Ma} Appendix). 

We choose a rank $1$ sublattice $I'\subset L$ with $(I, I')\not\equiv 0$. 
As explained, this induces a tube domain realization of ${\D}$ and a splitting of ${\GIQ}$. 
Dividing by ${\UIZ}$ and writing ${\XI}={\D}/{\UIZ}$, 
we obtain isomorphisms 
\begin{equation*}
{\XI} \simeq \mathcal{D}_{I}/{\UIZ}  \quad \subset \quad  {\DI}/{\UIZ} \simeq T(I), 
\end{equation*}
\begin{equation}\label{eqn: split GIQbar}
{\GIQbar} \: \simeq \: ({\rm O}^{+}({\UIZ})\times {\rm GL}(I))\ltimes {\UIQZ}, 
\end{equation}
both depending on the choice of $I'$. 
Then the canonical action of ${\GIZbar}$ on ${\XI}$ 
is induced from the natural action of 
$({\rm O}^{+}({\UIZ})\times {\rm GL}(I))\ltimes {\UIQZ}$ 
on $T(I)$ via these (non-canonical but compatible) isomorphisms. 
Here ${\rm O}^{+}({\UIZ})$ acts by torus automorphism fixing the identity, 
${\rm GL}(I)$ acts trivially, and 
${\UIQZ}$ acts by translation. 
The action of ${\GIZbar}$ on ${\XI}$ may have translation component, 
unless ${\GIZbar}$ is contained in the section of ${\rm O}^{+}({\UIZ})\times {\rm GL}(I)$.

\begin{remark}\label{remark: GlZ}
Let $I={\Z}l$ and $\Gamma(l)_{{\Z}}<{\GIZ}$ be the kernel of ${\GIZ}\to {\rm GL}(I)$. 
When $-{\id}\in {\G}$, we have ${\GIZ}=\Gamma(l)_{{\Z}}\times \{ \pm {\id} \}$, 
so we may replace ${\GIZ}$ by $\Gamma(l)_{{\Z}}$ when considering action on ${\D}$, 
as was done in \cite{Ma} Appendix.
(The last sentence of \cite{Ma} Remark A.8 for ${\G}={\Ost}$ is valid under the condition 
$\Gamma(l)_{{\Z}}={\GIZ}$ (e.g.~${\rm div}(I)>2$) or $A_{L}$ 2-elementary, or ${\rm div}(I)=1$. 
It is also valid when ${\rm div}(I)=2$ if $[l/2]\in A_{L}$ is indivisible.) 
\end{remark}

\subsection{Partial toroidal compactification}\label{ssec: toroidal cpt 0dim}

We recall partial toroidal compactification of ${\XI}={\D}/{\UIZ}$ following \cite{AMRT}. 
We put a ${\Q}$-structure on $U(I)_{{\R}}$ by ${\UIQ}\simeq L(I)_{{\Q}}$. 
We write ${\CII}={\CI}\cup \bigcup_{w}{\R}_{\geq 0}w$, 
where $w$ ranges over all isotropic vectors of $L(I)_{{\Q}}$ 
in the closure of ${\CI}$. 
A rational polyhedral cone decomposition (fan) $\Sigma=(\sigma_{\alpha})_{\alpha}$ 
in $U(I)_{{\R}}$ is called \textit{${\GIZ}$-admissible} (\cite{AMRT}) 
if the support of $\Sigma$ is ${\CII}$, 
$\Sigma$ is preserved under the adjoint (= natural) action of ${\GIZ}$ on $U(I)_{{\R}}=L(I)_{{\R}}$, 
and there are only finitely many cones up to the action of ${\GIZ}$. 
Isotropic rays in $\Sigma$ correspond to rational isotropic lines in $L(I)_{{\Q}}$ 
(hence independent of $\Sigma$), 
which in turn correspond to rank $2$ primitive isotropic sublattices $J$ of $L$ containing $I$. 

The fan $\Sigma$ defines a torus embedding $T(I)\hookrightarrow T(I)^{\Sigma}$ 
of the torus $T(I)=U(I)_{{\C}}/{\UIZ}$. 
Each ray $\sigma$ of $\Sigma$ 
defines a sub torus embedding $T(I)\hookrightarrow T(I)^{\sigma}\subset T(I)^{\Sigma}$, 
isomorphic to $({\C}^{\times})^{b} \hookrightarrow {\C}\times ({\C}^{\times})^{b-1}$, 
whose unique boundary divisor is the quotient torus defined by the quotient lattice 
${\UIZ}/({\R}\sigma \cap {\UIZ})$. 
The character group of this boundary torus is $\sigma^{\perp}\cap U(I)_{{\Z}}^{\vee}$. 
Here we regard $U(I)_{{\Z}}^{\vee}$ as a lattice on ${\UIQ}$ by the quadratic form on ${\UIQ}=L(I)_{{\Q}}$, 
which gives the pairing between ${\UIZ}$ and $U(I)_{{\Z}}^{\vee}$. 

We take a tube domain realization of ${\D}$ by choosing $I'\subset L$ with $(I, I')\not\equiv 0$. 
Then let ${\XIcpt}$ be the interior of the closure of 
${\XI}\simeq \mathcal{D}_{I}/{\UIZ}$ in $T(I)^{\Sigma}$. 
This embedding ${\XI}\hookrightarrow {\XIcpt}$ 
is the partial toroidal compactification over $I$ defined by the fan $\Sigma$. 
It is ${\GIZbar}$-equivariant, and does not depend on the choice of $I'$. 
We can think of ${\XIcpt}$ as giving a local chart for the boundary points 
of a full toroidal compactification lying over the $I$-cusp (see \S \ref{sec: toroidal cpt}), 
like ${\D}$ gives a local chart for the interior points in ${\G}\backslash {\D}$.


\section{Irregular $0$-dimensional cusps}\label{sec: irr 0dim cusp}

We now study irregular $0$-dimensional cusps. 
Let ${\G}$ be a finite-index subgroup of ${\OL}$ and 
$I$ be a rank $1$ primitive isotropic sublattice of $L$. 
We keep the notation from \S \ref{sec: 0dim cusp}. 
We will define irregularity in two stages: 
irregularity of a cusp (\S \ref{ssec: 0dim irregular define}), 
and irregularity of a toroidal boundary divisor over (or adjacent to) an irregular cusp (\S \ref{ssec: irregular boundary divisor 0dim}). 
The first stage is concerned only with ${\G}$, 
but the second stage is also involved with an ${\GIZ}$-admissible fan.

\subsection{Irregularity}\label{ssec: 0dim irregular define}

We give several equivalent definitions of irregularity of a $0$-dimensional cusp 
in the following form. 

\begin{proposition}\label{prop: characterize 0dim irregular}
The following conditions are equivalent. 
\begin{enumerate}
\item ${\UIZ}\ne {\UIZZ}$ where ${\UIZZ}={\UIQ}\cap {\GG}$. 
\item $-{\id}\not\in {\G}$ and $-E_{w}\in {\GIZ}$ for some $w\in L(I)_{{\Q}}$. 
\item $-{\id}\not\in {\G}$ and ${\GIZbar}\to {\Or}({\UIZ})$ is not injective. 
\item ${\GIZbar}$ contains an element which acts by a nonzero translation on ${\XI}={\D}/{\UIZ}$. 
\end{enumerate}
When these hold, we have 
${\UIZZ}/{\UIZ} = \langle E_{w} \rangle \simeq {\Z}/2$ and  
\begin{equation*}
{\ker}({\GIZbar}\to {\Or}({\UIZ})) = \langle -E_{w} \rangle \simeq {\Z}/2, 
\end{equation*}
and the translation in $(4)$ is given by $[w]\in {\UIQZ}$ and is unique. 
\end{proposition}

\begin{definition}
We say that the $0$-dimensional cusp $I$ is \textit{irregular} for ${\G}$ 
when these properties hold, 
and \textit{regular} otherwise. 
\end{definition}

\begin{proof}
$(1) \Rightarrow (2)$: 
We see that $-{\id}\not\in {\G}$ by ${\G}\ne {\GG}$. 
Let $E_{w}\in {\UIZZ}$ but $E_{w}\not\in {\UIZ}$. 
Since ${\GG}={\G}\sqcup -{\G}$, we have $E_{w}\in -{\G}$, 
and so $-E_{w}\in {\G}$. 
Note that ${\UIZZ}/{\UIZ}\simeq {\GG}/{\G}$ is of order $2$ and so 
${\UIZZ}/{\UIZ} = \langle E_{w} \rangle$. 

$(2) \Rightarrow (1)$: 
If $-E_{w}\in {\GIZ}$ and $-{\id}\not\in {\G}$, 
then $E_{w}\not\in{\UIZ}$ but $E_{w}\in {\UIZZ}$. 

$(2) \Rightarrow (3)$: 
Since $-E_{w}$ acts on $L(I)_{{\Q}}={\UIQ}$ trivially, 
its image in ${\GIZbar}$ is contained in the kernel of 
${\GIZbar}\to {\Or}({\UIZ})$. 

$(3) \Rightarrow (2)$: 
Recall that the kernel of ${\GIQbar}\to {\Or}({\UIZ})$ is 
\begin{equation*}
{\GL}(I)\times {\UIQZ} = {\UIQZ} \sqcup (-{\id})\cdot {\UIQZ}. 
\end{equation*}
Since ${\GIZbar}\cap {\UIQZ} = \{ 0 \}$, 
a nontrivial element of the kernel of ${\GIZbar}\to {\Or}({\UIZ})$ must be contained in $(-{\id})\cdot {\UIQZ}$, 
hence is the image of $-E_{w}$ for some $w\in L(I)_{{\Q}}$. 
This also shows that the kernel is ${\Z}/2$ generated by $-E_{w}$. 

$(2) \Rightarrow (4)$: 
The element $-E_{w}$ of ${\GIZbar}$ acts on ${\XI}$ by the translation by $[w] \in {\UIQZ}$. 
Since $-{\id}\not\in {\G}$, we have $E_{w}\not\in {\UIZ}$. 
This means that $[w]\ne 0 \in {\UIQZ}$. 

$(4) \Rightarrow (2), (3)$: 
We choose a splitting of ${\GIQbar}$ as in \eqref{eqn: split GIQbar} 
and express an element of 
${\GIZbar}\subset {\GIQbar}$ as 
$\gamma=(\gamma_{1}, \gamma_{2}, [w])$ accordingly, where 
$\gamma_{1}\in {\Or}({\UIZ})$, $\gamma_{2}\in {\rm GL}(I)$ and $[w]\in {\UIQZ}$. 
If $\gamma$ acts on ${\XI}$ by a nonzero translation, 
we must have $\gamma_{1}={\id}_{L(I)}$ and the translation is given by $[w]\in {\UIQZ}$.  
Therefore $\gamma$ is contained in the kernel of the projection to ${\Or}({\UIZ})$. 
Since ${\GIZbar}\cap {\UIQZ} = \{ 0 \}$ and $[w] \ne 0$, we have $\gamma_{2}=-{\id}_{I}$. 
Thus $\gamma=-E_{w}$. 
Finally, we have $-{\id}\not\in {\G}$, for otherwise 
$E_{w}=-\gamma$ would be contained in ${\UIZ}$ and then $[w]=0\in {\UIQZ}$. 
\end{proof}

\begin{remark}
Let $\Gamma(l)_{{\Z}}<{\GIZ}$ be as in Remark \ref{remark: GlZ}. 
By the condition (3), 
$I$ is irregular if and only if 
$-{\id}\not\in {\G}$, $\Gamma(l)_{{\Z}}\ne {\GIZ}$, and 
$\Gamma(l)_{{\Z}}$ and ${\GIZ}$ have the same image in ${\Or}({\UIZ})$. 
We do not use this characterization. 
\end{remark}

The condition (2) is useful for explicit calculation (\S \ref{sec: example}). 
We give some immediate consequences. 

\begin{corollary}
The group ${\G}$ has no irregular cusp 
when $-{\id}\in {\G}$ or when ${\G}$ is neat 
or when ${\G}<{\rm SO}^{+}(L)$ with $b$ odd. 
\end{corollary}

\begin{proof}
The case $-{\id}\in {\G}$ is obvious. 
When ${\G}$ is neat, the subquotient ${\GIZbar}$ is torsion-free, 
so it does not contain an element of finite order like $-E_{w}$. 
When $b$ is odd, $-E_{w}$ has determinant $(-1)^{b+2}=-1$,  
so a subgroup of ${\rm SO}^{+}(L)$ never contains such an element. 
\end{proof}

\begin{corollary}\label{cor: SO restrict} 
When $b$ is even, ${\G}$ has an irregular cusp if and only if 
${\G}\cap {\rm SO}^{+}(L)$ has an irregular cusp. 
\end{corollary}

\begin{proof}
When $b$ is even, both $-{\id}$ and $-E_{w}$ are contained in ${\rm SO}^{+}(L)$. 
\end{proof}

\begin{corollary}\label{cor: subgroup overgroup}
If ${\G}$ has an irregular cusp, 
any $\Gamma' < {\OL}$ with $\Gamma' \supset {\G}$ and $-{\id}\not\in \Gamma'$ has an irregular cusp. 
Equivalently, if $-{\id}\not\in {\G}$ and ${\G}$ has no irregular cusp, 
any subgroup of ${\G}$ of finite index has no irregular cusp. 
\end{corollary}

The lattice ${\UIZZ}={\UIQ}\cap {\GG}$ is the projection image of  
\begin{equation*}
U(I)_{{\Z}}^{\star} = (\{ \pm {\id} \} \cdot {\UIQ}) \cap {\G} 
= {\ker}({\GIZ}\to {\Or}({\UIZ})) 
\end{equation*} 
in ${\UIQ}$. 
Thus ${\UIZZ}$ is the lattice of translation in the ${\GIZ}$-action on the tube domain model. 
We have 
\begin{equation}\label{eqn: UIZstar}
U(I)_{{\Z}}^{\star}/{\UIZ} = 
\begin{cases}
\langle -{\id} \rangle \simeq {\Z}/2 & -{\id}\in {\G} \\ 
\{ 1 \} & -{\id}\not\in {\G}, I \; \textrm{regular} \\ 
\langle -E_{w} \rangle \simeq {\Z}/2 & I \; \textrm{irregular} 
\end{cases}
\end{equation}
This gives yet another characterization of irregularity: 
$-{\id}\not\in {\G}$ and ${\UIZ}\ne U(I)_{{\Z}}^{\star}$. 

As we will explain in \S \ref{ssec: modular form}, 
${\UIZ}$ is the lattice for Fourier expansion of ${\G}$-modular forms around $I$. 
Thus irregular $0$-dimensional cusps are those cusps 
whose lattice of translation is larger than the lattice for Fourier expansion.

\subsection{Irregular boundary divisors}\label{ssec: irregular boundary divisor 0dim}

Let $\Sigma=(\sigma_{\alpha})$ be a ${\GIZ}$-admissible fan in $U(I)_{{\R}}$, 
and ${\XI}\hookrightarrow {\XIcpt}$ be the partial compactification 
defined in \S \ref{ssec: toroidal cpt 0dim}. 
For a ray $\sigma$ in $\Sigma$ we denote by $D(\sigma)\subset {\XIcpt}$ 
the corresponding boundary divisor. 
When $I$ is irregular, these boundary divisors are divided into two types as follows. 

\begin{proposition}\label{prop: irregular ray}
Let $I$ be an irregular $0$-dimensional cusp for ${\G}$. 
Let $-E_{w}\in {\GIZ}$. 
The following conditions for a ray $\sigma$ in $\Sigma$ are equivalent: 
\begin{enumerate}
\item $\sigma \cap {\UIZ} \ne \sigma \cap {\UIZZ}$. 
\item $-E_{w}$ acts trivially on the boundary divisor $D(\sigma)$. 
\item $D(\sigma)$ is fixed by some nontrivial element of ${\GIZbar}$. 
\end{enumerate}
When these hold, the element in $(3)$ is given by $-E_{w}$. 
In particular, it is unique, independent of $\sigma$, and of order $2$. 
\end{proposition}

\begin{definition}\label{def: irregular ray}
When these properties hold, we call $\sigma$ an \textit{irregular ray} and 
$D(\sigma)$ an \textit{irregular boundary divisor}. 
Otherwise we call $\sigma$ \textit{regular}. 
For the sake of completeness, we call any ray $\sigma$ \textit{regular} when $I$ is regular. 
\end{definition}

\begin{proof}
$(1) \Leftrightarrow (2)$: 
Recall that $-E_{w}$ acts on ${\XI}\subset T(I)$ as the translation by $[w]\in {\UIQZ}$. 
A Zariski open set of $D(\sigma)$ is the quotient torus (or its analytic open set) 
defined by the quotient lattice ${\UIZ}/\Lambda_{\sigma}$ 
where $\Lambda_{\sigma}= {\R}\sigma \cap {\UIZ}$. 
Hence $-E_{w}$ acts on $D(\sigma)$ as the translation by 
the image of $[w]$ in ${\UIQ}/({\UIZ}+(\Lambda_{\sigma})_{{\Q}})$. 
This is trivial if and only if 
$w\in {\UIZ}+(\Lambda_{\sigma})_{{\Q}}$, 
which in turn is equivalent to 
$\Lambda_{\sigma} \ne {\R}\sigma \cap {\UIZZ}$. 
In this case $-E_{w}$ acts by $-1$ on the normal torus 
$(\Lambda_{\sigma})_{{\C}}/\Lambda_{\sigma}\simeq {\C}^{\times}$. 

$(2) \Rightarrow (3)$ is obvious. 

$(3) \Rightarrow (2)$: 
Suppose that $\gamma\in {\GIZbar}$ acts trivially on $D(\sigma)$. 
Let $\gamma_{1}$ be the image of $\gamma$ in ${\Or}({\UIZ})$. 
Then $\gamma_{1}$ must preserve $\sigma\cap {\UIZ}$ and act trivially on ${\UIZ}/\Lambda_{\sigma}$. 
Hence $\gamma_{1}$ acts trivially on $\Lambda_{\sigma}$ and $\Lambda_{\sigma}^{\perp}$, 
and so $\gamma_{1}={\id}$. 
This implies that $\gamma$ is contained in the kernel of ${\GIZbar}\to {\Or}({\UIZ})$, 
whence $\gamma=-E_{w}$ by Proposition \ref{prop: characterize 0dim irregular}. 
Therefore $-E_{w}$ acts trivially on $D(\sigma)$. 
\end{proof}

\begin{corollary}\label{cor: irregular 0dim ramify}
When $I$ is irregular, the quotient map ${\XIcpt}\to {\XIcpt}/{\GIZbar}$ 
is ramified along the irregular boundary divisors with ramification index $2$, 
caused by the common subgroup $\langle -E_{w} \rangle \simeq {\Z}/2$ of ${\GIZbar}$,  
and not ramified along other boundary divisors. 
When $I$ is regular, ${\XIcpt}\to {\XIcpt}/{\GIZbar}$ is not ramified along any boundary divisor. 
\end{corollary}

\begin{proof}
It remains to supplement the argument in the case $I$ is regular. 
If $\gamma \in {\GIZbar}$ fixes a boundary divisor, 
we see that $\gamma$ acts trivially on ${\UIZ}$ 
by the same argument as $(3)\Rightarrow (2)$ above. 
When $-{\id}\not\in {\G}$,  
${\GIZbar}\to {\Or}({\UIZ})$ is injective by the condition (3) of Proposition \ref{prop: characterize 0dim irregular}, 
so we find that $\gamma= {\id}$. 
When $-{\id}\in {\G}$, the kernel of ${\GIZbar}\to {\Or}({\UIZ})$ is $\{ \pm {\id} \}$, 
so $\gamma=\pm {\id}$, which acts trivially on ${\XI}$. 
\end{proof}


\section{Examples}\label{sec: example}

In this section we study some examples of groups with/without irregular cusp. 
Logically this section should be read after \S \ref{sec: irr 1dim cusp} 
where we complete the discussion of irregular $1$-dimensional cusps. 
But we encourage the reader to read this section just after \S \ref{sec: irr 0dim cusp} 
for the following two reasons.  
First, most part of \S \ref{sec: 1dim cusp} and \S \ref{sec: irr 1dim cusp} is designed for 
\S \ref{sec: toroidal cpt} and \ref{sec: cusp form criterion}, 
while the only result from \S \ref{sec: 1dim cusp} and \S \ref{sec: irr 1dim cusp} 
we need in this section is 
Corollary \ref{cor: 1dim reduced to 0dim}, 
which just says that 
${\G}$ has no irregular $1$-dimensional cusp if it has no irregular $0$-dimensional cusp. 
Second, it is Proposition \ref{prop: characterize 0dim irregular} (2) that is frequently used in this section, 
so we do not want to put this section too far from it. 

\textit{We assume that the lattice $L$ is even in this section.} 
The quotient $A_{L}=L^{\vee}/L$ is called the \textit{discriminant group} of $L$,  
equipped with a canonical quadratic form $A_{L}\to {\Q}/2{\Z}$ 
called the \textit{discriminant form}. 
If $I$ is a rank $1$ primitive isotropic sublattice of $L$, 
we write ${\rm div}(I)$ for the positive generator of the ideal $(I, L)\subset {\Z}$. 
Then $I^{\ast}={\rm div}(I)^{-1}I$ is primitive in $L^{\vee}$, 
and we have a canonical isometry 
$I^{\perp}\cap L^{\vee}/I^{\ast} \simeq (I^{\perp}/I)^{\vee}$.

\subsection{Stable orthogonal groups}\label{ssec: stable orthogonal}

Let $L$ be an even lattice of signature $(2, b)$. 
Let ${\Ost}<{\OL}$ be the kernel of the reduction map ${\OL}\to {\rm O}(A_{L})$, 
called the \textit{stable orthogonal group} or the \textit{discriminant kernel}. 
The following was asserted in \cite{Ma} p.901. 
We supplement the proof for the sake of completeness. 

\begin{lemma}\label{lem: UIZ=LI}
Let $I$ be a rank $1$ primitive isotropic sublattice of $L$. 
For ${\G}={\Ost}$ we have ${\UIZ}=L(I)$.  
\end{lemma}

\begin{proof}
We take a generator $l$ of $I$. 
The inclusion $L(I)\subset {\UIZ}$ can be checked by testing   
the definition of $E_{m\otimes l}(v)$ for $v\in L^{\vee}$ and $m\in I^{\perp}/I$, 
taking a lift of $m$ from $I^{\perp}\cap L$. 
Conversely, if $E_{m\otimes l}\in {\Ost}$ 
for a vector $m\in I^{\perp}_{{\Q}}/I_{{\Q}}$, 
then $E_{m\otimes l}(v) = v - (m, v)l$ must be contained in $v+L$ for $v\in I^{\perp}\cap L^{\vee}$. 
This implies that $(m, v)\in {\Z}$ for every $v\in I^{\perp}\cap L^{\vee}$, 
and so $m\in (I^{\perp}/I)^{\vee \vee}=I^{\perp}/I$.  
%
\end{proof}

We obtain a first example of regular cusps. 

\begin{lemma}\label{lem: regular div1}
If ${\G}\supset {\Ost}$ and ${\rm div}(I)=1$, 
then $I$ is a regular cusp for ${\G}$. 
\end{lemma}

\begin{proof}
We take a generator $l$ of $I$. 
Since ${\rm div}(I)=1$, we can take an isotropic vector $l'\in L$ with $(l, l')=1$. 
We can and do identify $I^{\perp}/I$ with $\langle l, l' \rangle^{\perp}\cap L$. 
We have the splitting $L=\langle l, l' \rangle \oplus (\langle l, l' \rangle^{\perp}\cap L)$.  
Suppose $-E_{m\otimes l}\in {\G}$ for a vector 
$m\in \langle l, l' \rangle^{\perp}\cap L_{{\Q}}$. 
Then $E_{m\otimes l}$ preserves $L$, so we find that the vector 
$E_{m\otimes l}(l') = l' + m - \frac{1}{2}(m, m)l$ 
is contained in $L$. 
This implies that $m\in \langle l, l' \rangle^{\perp}\cap L=I^{\perp}/I$. 
Since ${\G}\supset {\Ost}$, we have ${\UIZ}\supset L(I)$ by Lemma \ref{lem: UIZ=LI}, 
and so $m\otimes l \in {\UIZ}$. 
This means that $E_{m\otimes l}\in {\G}$, 
and then $-{\id}\in {\G}$. 
\end{proof}

For ${\G}={\Ost}$ we have the following constraints for existence of irregular cusp.  

\begin{lemma}\label{lem: constraint irregular stable}
If $I$ is an irregular $0$-dimensional cusp for ${\G}={\Ost}$, 
then ${\rm div}(I)=2$, $A_{L(I)}$ is 2-elementary, and 
${\UIZZ}/{\UIZ}\simeq {\Z}/2$ is a subgroup of $A_{L(I)}$. 
\end{lemma}

\begin{proof}
Let $I={\Z}l$ and assume that 
$-E_{m\otimes l}\in {\Ost}$ for a vector $m$ of $I^{\perp}_{{\Q}}/I_{{\Q}}$. 
Then for any $v\in I^{\perp}\cap L^{\vee}$ 
the vector $-E_{m\otimes l}(v)=-v+(m, v)l$ must be contained in $v+L$. 
This implies that  
\begin{equation}\label{eqn: constraint irregular stable}
2v \in (m, v)l + L.  
\end{equation}
If we substitute $v=l/{\rm div}(I)$, 
we find that $(2/{\rm div}(I))l \in L$, and so ${\rm div}(I)=1$ or $2$. 
The case ${\rm div}(I)=1$ is excluded by Lemma \ref{lem: regular div1}. 
Thus ${\rm div}(I)=2$. 

If $[v]\in (I^{\perp}/I)^{\vee}$ denotes the image of $v\in I^{\perp}\cap L^{\vee}$, 
then \eqref{eqn: constraint irregular stable} means that $2[v]\in I^{\perp}/I$. 
This shows that $A_{I^{\perp}/I}$ is 2-elementary. 
Finally, \eqref{eqn: constraint irregular stable}  
implies that $(m, v)\in {\Z}$ for every $v\in I^{\perp}\cap L$, 
and so $m\in (I^{\perp}/I)^{\vee}$. 
\end{proof}

This determines the structure of $A_{L}$ when ${\Ost}$ has an irregular $0$-dimensional cusp. 

\begin{proposition}\label{prop: AL irregular stable}
If ${\Ost}$ has an irregular cusp, then 
$A_{L}\simeq {\Z}/8 \oplus ({\Z}/2)^{\oplus a}$ or 
$A_{L}\simeq ({\Z}/4)^{\oplus 2} \oplus ({\Z}/2)^{\oplus a}$ 
as abelian groups. 
\end{proposition}

\begin{proof}
Let $I={\Z}l$ be as in Lemma \ref{lem: constraint irregular stable}. 
Let $x=[l/2]\in A_{L}$. 
Since $A_{L(I)}\simeq x^{\perp}/x$ is 2-elementary, 
$A_{L}$ must be isomorphic to either 
$({\Z}/2)^{\oplus a}$ or ${\Z}/4 \oplus ({\Z}/2)^{\oplus a}$ or 
${\Z}/8 \oplus ({\Z}/2)^{\oplus a}$ or $({\Z}/4)^{\oplus 2} \oplus ({\Z}/2)^{\oplus a}$. 
The first case $A_{L}\simeq ({\Z}/2)^{\oplus a}$ cannot occur because then $-{\id}\in {\Ost}$. 
The second case $A_{L}\simeq {\Z}/4 \oplus ({\Z}/2)^{\oplus a}$ does not occur too, 
because then $x^{\perp}/x$ would contain an element of order $4$. 
\end{proof}

\begin{remark}
Further calculation shows that 
$x=[l/2]\in A_{L}$ is divisible by $4$ (and hence unique) in the case 
$A_{L}\simeq {\Z}/8 \oplus ({\Z}/2)^{\oplus a}$, 
and divisible by $2$ in the case $A_{L}\simeq ({\Z}/4)^{\oplus 2} \oplus ({\Z}/2)^{\oplus a}$. 
\end{remark}

\begin{example}\label{ex: K3}
Let $L=2U \oplus mE_{8} \oplus \langle -2d \rangle$. 
Then ${\Ost}$ has no irregular cusp when $d\ne 4$. 
We show in Proposition \ref{prop: irregular example -8} that 
${\Ost}$ indeed has an irregular cusp when $d=4$. 
When $m=2$, ${\Ost}$ is the modular group for the moduli space of polarized $K3$ surfaces of degree $2d$. 
\end{example}

\begin{example}\label{ex: K3[N]}
Let $L=2U \oplus mE_{8} \oplus \langle -2t \rangle \oplus \langle -2d \rangle$. 
Then ${\Ost}$ has no irregular cusp when 
$(t, d) \ne (4, 1), (2, 2), (1, 4)$. 
We show in \S \ref{ssec: irregular example} that 
${\Ost}$ indeed has an irregular cusp in these exceptional cases. 
When $m=2$, ${\Ost}$ is the modular group for the moduli space of 
polarized irreducible symplectic manifolds of $K3^{[t-1]}$-type 
with polarization of split type and degree $2d$ (\cite{GHS10}). 
\end{example}

\begin{example}
When $L=U\oplus 2E_{8} \oplus M$, 
where $M$ is a certain lattice of signature $(1, 2)$ and 
discriminant $d\equiv 2$ mod $6$,  
${\Ost}$ is the modular group for the moduli space of 
special cubic fourfolds of discriminant $d$ (\cite{TVA}). 
Since $A_{L}$ has length $\leq 3$ and order $d$, 
we find that ${\Ost}$ has no irregular cusp when $d \ne 8, 32$. 
\end{example} 

\begin{example}
Similarly, when $L=U\oplus 2E_{8} \oplus M$, 
where $M$ is a certain lattice of signature $(1, 2)$ and 
discriminant $d\equiv 0, 2, 4$ mod $8$, 
${\Ost}$ is the modular group for the moduli space of 
special $K3^{[2]}$-fourfolds of degree $2$ and discriminant $d$ (\cite{Pe}). 
This group has no irregular cusp when $d\ne 32$. 
\end{example}

\begin{example}
When $L=U\oplus 2E_{8}\oplus \langle 2d \rangle$, 
${\Ost}$ is the modular group for the moduli space of 
$U\oplus \langle -2d \rangle$-polarized $K3$ surfaces studied in \cite{FM}. 
This group has no irregular cusp when $d\ne 4$. 
\end{example}

\subsection{O'Grady 10}\label{ssec: OG10}

In this subsection we let $L$ be an even lattice of the form 
$L= M \oplus \langle -2d \rangle$ with 
$M$ of signature $(2, b-1)$. 
We consider the group 
\begin{equation*}
{\G} = \{ \: \gamma \in {\OL} \: | \: 
\gamma|_{A_{M}}=\pm{\id}, \: \gamma|_{A_{\langle -2d \rangle}}={\id} \; \}. 
\end{equation*}
Then ${\G}$ contains ${\Ost}$ with index $\leq 2$, 
with ${\G}={\Ost}$ if and only if $A_{M}$ is 2-elementary. 
We have $-{\id}\in {\G}$ if and only if $d=1$. 
When $M=2U\oplus 2E_{8} \oplus A_{2}$, 
${\G}$ is the modular group for the moduli space of polarized O'Grady $10$ manifolds 
with polarization of split type and degree $2d$ (\cite{GHS11}). 

\begin{proposition}\label{prop: OG10}
The group ${\G}$ has no irregular cusp when $d\ne 2, 4$. 
\end{proposition}

\begin{proof}
Assume that $I={\Z}l$ is an irregular cusp for ${\G}$ 
and $-E_{m\otimes l}\in {\G}$ for $m\in I^{\perp}_{{\Q}}/I_{{\Q}}$. 
The case $d=1$ is excluded by $-{\id}\not\in {\G}$. 
We shall show that $d\mid 4$. 
Since 
$E_{2m\otimes l}=(-E_{m\otimes l})\circ (-E_{m\otimes l}) \in {\Ost}$, 
we see that $2m\otimes l \in L(I)$ by Lemma \ref{lem: UIZ=LI}. 
Hence we can take a lift $\tilde{m}$ of $m$ from $I^{\perp}\cap \frac{1}{2}L$. 
Let $v$ be a generator of $\langle -2d \rangle^{\vee}$. 
The vector  
\begin{equation*}
-E_{m\otimes l}(v) = -v + (\tilde{m}, v)l - (l, v)\tilde{m} + \frac{1}{2}(m, m)(l, v)l  
\end{equation*}
must be contained in $v+L$, 
and hence 
\begin{equation}\label{eqn: OG10 calculation}
2v \in (\tilde{m}, v)l - (l, v)\tilde{m} + \frac{1}{2}(m, m)(l, v)l + L.
\end{equation}
Since $(\tilde{m}, v)\in \frac{1}{2}{\Z}$, 
$(l, v)\in {\Z}$ and $(m, m)\in \frac{1}{2}{\Z}$, 
we find that 
$2v\in \frac{1}{4}L$. 
Hence $2d\mid 8$. 
\end{proof}

The case $d=2$ does not occur when $|A_{M}|$ is square-free, 
because then $A_{L}$ is anisotropic and hence ${\rm div}(I)=1$. 
In Proposition \ref{prop: irregular example -8} 
we show that ${\G}$ indeed has an irregular cusp when $d=4$ and $M$ contains $U$.

\subsection{Generalized Kummer}\label{ssec: gene Kummer}

In this subsection we let $L=M\oplus \langle -2d \rangle$ be as in \S \ref{ssec: OG10} 
and consider the group 
\begin{equation*}
{\G} = \{ \: \gamma \in {\OL} \: | \: 
\gamma|_{A_{M}}=\det (\gamma) {\id}, \: \gamma|_{A_{\langle -2d \rangle}}={\id} \; \}. 
\end{equation*}
This is an index $\leq 2$ subgroup of the group considered in \S \ref{ssec: OG10}. 
When $M=2U \oplus \langle -2t \rangle$ with $t\geq 3$, 
${\G}$ is the modular group for the moduli space of 
polarized deformation generalized Kummer varieties of $A^{[t]}$-type 
with polarization of split type and degree $2d$ (\cite{Da}). 

\begin{proposition}
The group ${\G}$ has no irregular cusp when $d\nmid 4$. 
Moreover, when $b$ is even, ${\G}$ has no irregular cusp unless when 
$A_{L}$ is isomorphic to 
${\Z}/8 \oplus ({\Z}/2)^{\oplus a}$ or $({\Z}/4)^{\oplus 2} \oplus ({\Z}/2)^{\oplus a}$ 
as abelian groups. 
\end{proposition}

\begin{proof}
The assertion $d\nmid 4$ follows from Corollary \ref{cor: subgroup overgroup} and Proposition \ref{prop: OG10}. 
Since ${\G}\cap {\rm SO}^{+}(L) = {\Ost}\cap {\rm SO}^{+}(L)$, 
Corollary \ref{cor: SO restrict} shows that 
when $b$ is even, ${\G}$ has an irregular cusp if and only if ${\Ost}$ has an irregular cusp. 
Then our assertion follows from Proposition \ref{prop: AL irregular stable}. 
\end{proof}

This shows that 
when $M=2U \oplus \langle -2t \rangle$, 
${\G}$ has no irregular cusp when $(t, d)\ne (4, 1), (2, 2), (1, 4)$. 
In \S \ref{ssec: irregular example} we show that 
${\G}$ indeed has an irregular cusp in these exceptional cases.

\subsection{Special cubic fourfolds}

In this subsection we let $L$ be an even lattice of the form 
$L=M\oplus K$ with $|A_{K}|>1$ odd. 
($K$ may be either negative-definite or hyperbolic or of signature $(2, \ast)$.) 
We consider the group 
\begin{equation*}
{\G} = \{ \: \gamma \in {\OL} \: | \: 
\gamma|_{A_{M}}= \pm{\id}, \: \gamma|_{A_{K}}={\id} \; \}. 
\end{equation*}
When $M=\langle 2n \rangle \oplus U \oplus 2E_{8}$ and $K=A_{2}$, 
${\G}$ is the modular group for the moduli space of special cubic fourfolds of discriminant $6n$ (\cite{TVA}). 

\begin{proposition}
The group ${\G}$ has no irregular cusp. 
\end{proposition}

\begin{proof}
Suppose to the contrary that 
$I={\Z}l$ is an irregular cusp and $-E_{m\otimes l}\in {\G}$. 
As in the proof of Proposition \ref{prop: OG10}, 
we can take a lift of $m$ from $I^{\perp}\cap \frac{1}{2}L$. 
We take a vector $v\in K^{\vee}-K$. 
Then $-E_{m\otimes l}(v)$ must be contained in $v+L$. 
The same calculation as \eqref{eqn: OG10 calculation} tells us that 
$8v\in L$. 
Therefore $[v]\in A_{K}\subset A_{L}$ satisfies $8[v]=0$, 
but this contradicts the assumption that $|A_{K}|$ is odd. 
\end{proof}

\subsection{Examples of irregular cusps}\label{ssec: irregular example}

In this subsection we present two series of examples of irregular cusps, infinitely many in every dimension. 
We will denote by $e, f$ the standard basis of $U$. 

As the first series of examples, we consider even lattices of the form 
$L=U\oplus \langle -8 \rangle \oplus M$ with $M$ hyperbolic. 
We define the group ${\G}$ by 
\begin{equation*}
{\G} = \{ \: \gamma \in {\OL} \: | \: 
\gamma|_{A_{M}}=\pm{\id}, \: \gamma|_{A_{\langle -8 \rangle}}={\id} \; \}. 
\end{equation*}
This is the group considered in \S \ref{ssec: OG10} with $d=4$. 
The group ${\G}$ contains ${\Ost}$ with index $\leq 2$, 
and we have ${\G}={\Ost}$ if and only if $A_{M}$ is 2-elementary. 

\begin{proposition}\label{prop: irregular example -8}
The group ${\G}$ has an irregular cusp. 
\end{proposition}

\begin{proof}
First note that $-{\id}\not\in {\G}$ by the condition $\gamma|_{A_{\langle -8 \rangle}}={\id}$. 
Let $v$ be a generator of $\langle -8 \rangle$. 
We take the vectors 
\begin{equation*}
l = 2e + 2f + v, \quad 
m = e/2 - f/2, 
\end{equation*}
and show that 
$-E_{m\otimes l}\in {\G}$. 
This amounts to checking the following: 
\begin{equation*}
E_{m\otimes l}(L) \subset L, \quad 
E_{m\otimes l}|_{A_{M}} = \pm{\id}, \quad 
E_{m\otimes l}(v/8) \in -v/8+L. 
\end{equation*}
Since $M\perp \langle l, m \rangle$, 
$E_{m\otimes l}$ acts trivially on $M$. 
By direct calculation, we see that 
\begin{equation*}
E_{m\otimes l}(e)=4e+f+v, \quad 
E_{m\otimes l}(f)=e,  \quad 
E_{m\otimes l}(v)=-v-8e. 
\end{equation*} 
This proves our assertion. 
\end{proof}

As the second series of examples, 
we consider even lattices of the form 
$L=U\oplus \langle -4 \rangle^{\oplus 2} \oplus M$ 
with $M$ hyperbolic, 
and the group ${\G}$ defined by 
\begin{equation*}
{\G} = \{ \: \gamma \in {\OL} \: | \: 
\gamma|_{A_{M}}=\pm{\id}, \: \gamma|_{A_{\langle -4 \rangle^{\oplus 2}}}={\id} \; \}. 
\end{equation*}
The group ${\G}$ contains ${\Ost}$ with index $\leq 2$,  
and ${\G}={\Ost}$ if and only if $A_{M}$ is 2-elementary. 

\begin{proposition}\label{prop: irregular example -4+-4}
The group ${\G}$ has an irregular cusp. 
\end{proposition}

\begin{proof}
This is similar to the first example. 
We let $v_{1}, v_{2}$ be the standard basis of $\langle -4 \rangle^{\oplus 2}$ 
and show that $-E_{m\otimes l}\in {\G}$ 
for the vectors 
\begin{equation*}
l = 2e + 2f + v_{1} + v_{2}, \quad 
m = e + v_{1}/2. 
\end{equation*}
The detail is left to the reader. 
\end{proof}


\section{$1$-dimensional cusps}\label{sec: 1dim cusp}

In this section we recall, following \cite{Sc}, \cite{Ko}, \cite{GHS07}, \cite{Lo},  
the structure of the stabilizer of a $1$-dimensional cusp of ${\D}={\DL}$ 
with its action on the Siegel domain model,   
and the canonical partial toroidal compactification over the cusp. 
This is a long preliminary for the next \S \ref{sec: irr 1dim cusp}. 
Although this section is no more than expository, 
we need to keep the rather self-contained style of \S \ref{sec: 0dim cusp}, 
for the same reasons as in \S \ref{sec: 0dim cusp} and for consistency. 

Throughout this section we fix a rank $2$ primitive isotropic sublattice $J$ of $L$. 
The choice of the component ${\D}$ determines 
a connected component of ${\proj}J_{{\C}} - {\proj}J_{{\R}}$, 
which is the cusp corresponding to $J$. 
This in turn determines an orientation of $J$. 
We write 
\begin{equation*}
L(J) = J^{\perp}\cap L/J, 
\end{equation*}
which is a negative-definite lattice of rank $b-2$. 
We will call an embedding $2U_{{\Q}}\hookrightarrow L_{{\Q}}$ a \textit{splitting for} $J_{{\Q}}$ 
if it sends the standard $2$-dimensional isotropic subspace of $2U_{{\Q}}$ to $J_{{\Q}}$. 
This defines a lift $L(J)_{{\Q}}\hookrightarrow J^{\perp}_{{\Q}}$ of $L(J)_{{\Q}}$ as $2U_{{\Q}}^{\perp}$.

\subsection{Siegel domain model}\label{ssec: Siegel domain}

We consider restriction of the two-step projection 
${\proj}L_{{\C}} \dashrightarrow {\proj}(L/J)_{{\C}} \dashrightarrow {\proj}(L/J^{\perp})_{{\C}}$ 
to ${\D}\subset Q \subset {\proj}L_{{\C}}$. 
Since $Q$ contains the line ${\proj}J_{{\C}}$, 
the center of the first projection, 
a plane containing ${\proj}J_{{\C}}$ either intersects with $Q$ at two lines (one ${\proj}J_{{\C}}$) 
or is contained in $Q$. 
The latter occurs when the plane is contained in ${\proj}J^{\perp}_{{\C}}$. 
This shows that the first projection 
\begin{equation*}
\pi_{1} : Q - Q\cap {\proj}J^{\perp}_{{\C}} \to {\proj}(L/J)_{{\C}} - {\proj}L(J)_{{\C}} 
\end{equation*}
is an affine line bundle. 
We identify $(L/J^{\perp})_{{\C}}=J^{\vee}_{{\C}}$ by the pairing. 
The second projection 
\begin{equation*}
\pi_{2} : {\proj}(L/J)_{{\C}} - {\proj}L(J)_{{\C}} \to 
{\proj}(L/J^{\perp})_{{\C}} = {\proj}J^{\vee}_{{\C}} 
\end{equation*}
is an affine space bundle isomorphic to 
$L(J)_{{\C}} \otimes \mathcal{O}_{{\proj}J^{\vee}_{{\C}}}(1)$. 
Here, by abuse of notation, 
$\mathcal{O}_{{\proj}J^{\vee}_{{\C}}}(1)$ stands for the line bundle corresponding to this sheaf 
(the dual of the tautological line bundle). 
A choice of a splitting for $J_{{\C}}$ determines a section of $\pi_{2}$. 

The orientation of $J$ determines a connected component ${\HJ}$ of 
${\proj}J^{\vee}_{{\C}} - {\proj}J^{\vee}_{{\R}}$. 
We write 
${\VJ} = \pi_{2}^{-1}({\HJ})$ and  
${\D}(J) = \pi_{1}^{-1}({\VJ})$. 
By definition, ${\D}(J)$ consists of ${\C}\omega\in Q$ such that 
the map $(\cdot , \omega)\colon J_{{\R}}\to {\C}$ 
is an orientation-preserving ${\R}$-isomorphism. 
We thus have the enlarged two-step fibration for ${\D}$: 
\begin{equation*}
{\D} \subset {\D}(J) \stackrel{\pi_{1}}{\to} {\VJ} \stackrel{\pi_{2}}{\to} {\HJ}. 
\end{equation*}
This is the Siegel domain realization of ${\D}$ with respect to $J$.  
Here 
${\D}(J) \to {\VJ}$ is an affine line bundle, 
inside which ${\D}\to {\VJ}$ is an upper half space bundle. 
Over ${\HJ}\subset {\proj}J_{{\C}}^{\vee}$ we have the Hodge bundles 
\begin{equation*}
(F^{\vee})^{1,0} = \mathcal{O}_{{\HJ}}(-1) \subset \underline{J_{{\C}}^{\vee}}, \quad  
F^{1,0} = ((F^{\vee})^{1,0})^{\perp} \subset \underline{J_{{\C}}},  
\end{equation*}
where we write 
$\underline{J_{{\C}}^{\vee}} = J_{{\C}}^{\vee} \otimes \mathcal{O}_{{\HJ}}$ and 
$\underline{J_{{\C}}} = J_{{\C}} \otimes \mathcal{O}_{{\HJ}}$.  
Then  
$\mathcal{O}_{{\HJ}}(1)\simeq \underline{J_{{\C}}}/F^{1,0}$ 
canonically. 
This shows that ${\VJ}\to {\HJ}$ is an affine space bundle isomorphic to 
$L(J)_{{\C}} \otimes (\underline{J_{{\C}}}/F^{1,0})$. 

The relationship with the tube domain model is as follows. 
We choose a rank $1$ primitive sublattice $I$ of $J$. 
This corresponds to a $0$-dimensional cusp in the closure of the $1$-dimensional cusp for $J$. 
The filtration $I\subset J \subset J^{\perp} \subset L$ determines the projections 
${\proj}(L/I)_{{\C}}\dashrightarrow {\proj}(L/J)_{{\C}}\dashrightarrow {\proj}(L/J^{\perp})_{{\C}}$. 
Then the composition of this with the tube domain realization 
${\D}\subset {\DI} \hookrightarrow {\proj}(L/I)_{{\C}}$ 
is the Siegel domain realization above.

\subsection{Stabilizer over ${\Q}$}\label{ssec: stab Q 1dim cusp}

Let ${\GJQ}$ be the subgroup of 
the stabilizer of $J_{{\Q}}$ in ${\rm O}^{+}(L_{{\Q}})$ 
that acts on $J_{{\Q}}$ with determinant $1$. 
The determinant $1$ condition is not restrictive when restricting to subgroups of ${\OL}$. 
We write 
\begin{equation*}
{\WJQ} = {\ker}({\GJQ} \to {\rm O}(L(J)_{{\Q}}) \times {\rm SL}(J_{{\Q}})), 
\end{equation*}
\begin{equation*}
{\UJQ} = {\ker}({\GJQ} \to {\rm GL}(J_{{\Q}}^{\perp})), 
\end{equation*}
\begin{equation*}
{\VJQ} = {\WJQ}/{\UJQ}. 
\end{equation*}
By definition we have the canonical exact sequences 
\begin{equation}\label{eqn: WJQ sequence}
1 \to {\WJQ} \to {\GJQ} \to {\rm O}(L(J)_{{\Q}}) \times {\rm SL}(J_{{\Q}}) \to 1, 
\end{equation}
\begin{equation*}
0 \to {\UJQ} \to {\WJQ} \to {\VJQ} \to 0. 
\end{equation*}
The group ${\WJQ}$ is the unipotent part of ${\GJQ}$. 
The first sequence splits \textit{non-canonically} 
if we choose a splitting for $J_{{\Q}}$, 
while the second never splits, being a Heisenberg group. 

We have a $\wedge^{2}J$-valued symplectic form on $L(J)\otimes J$ 
induced from the quadratic form on $L(J)$ 
and the canonical symplectic form $J\times J \to \wedge^{2}J$. 
This gives a Heisenberg group structure on 
$\wedge^{2}J_{{\Q}}\times (L(J)_{{\Q}}\otimes J_{{\Q}})$. 
Explicitly, 
we identify $\wedge^{2}J_{{\Q}}\simeq {\Q}$ 
and take a bijection 
$L(J)_{{\Q}}\otimes J_{{\Q}} \simeq L(J)_{{\Q}} \times L(J)_{{\Q}}$ 
by choosing a positive basis of $J$, 
and define product on 
${\Q}\times L(J)_{{\Q}} \times L(J)_{{\Q}}$ by 
\begin{equation*}
(\alpha, v_{1}, v_{2})\cdot (\beta, w_{1}, w_{2}) = 
(\alpha+\beta+(v_{2}, w_{1}), v_{1}+w_{1}, v_{2}+w_{2}). 
\end{equation*}
The center is ${\Q}\times \{ 0 \} \times \{ 0 \}$. 

\begin{lemma}[cf.~\cite{Lo}]
${\WJQ}$ is isomorphic to the Heisenberg group for $L(J)_{{\Q}}\otimes J_{{\Q}}$ 
with center ${\UJQ}$, 
and we have the canonical isomorphisms 
\begin{equation*}
\wedge^{2}J_{{\Q}} \to {\UJQ}, \quad l\wedge l' \mapsto E_{l\otimes l'},  
\end{equation*}
\begin{equation*}
L(J)_{{\Q}}\otimes J_{{\Q}} \to {\VJQ}, \quad m\otimes l \mapsto E_{\tilde{m}\otimes l} \; {\rm mod} \; {\UJQ}. 
\end{equation*}
\end{lemma}

\begin{proof}
We choose a rank $1$ primitive sublattice $I$ of $J$ and put 
$\bar{J}=(J/I)\otimes I \subset L(I)$. 
Note that $\bar{J}\simeq \wedge^{2}J$ naturally. 
We restrict the sequence \eqref{eqn: canonical sequence GIQ} for ${\GIQ}$ 
to ${\WJQ}\subset {\GIQ}$. 
It is clear that  
${\WJQ} \cap {\UIQ} = \bar{J}_{{\Q}}^{\perp}\cap L(I)_{{\Q}}$, 
which contains ${\UJQ}$ with  
\begin{equation}\label{eqn: UJQ in UIQ}
{\UJQ} = (\bar{J}_{{\Q}}^{\perp})^{\perp} = 
\bar{J}_{{\Q}} \simeq \wedge^{2}J_{{\Q}} \quad \subset {\UIQ}. 
\end{equation} 
The image of 
${\WJQ}\to {\Or}(L(I)_{{\Q}})$ 
is the subgroup of the stabilizer of $\bar{J}_{{\Q}}$ 
that acts trivially on $\bar{J}_{{\Q}}$ and 
$\bar{J}_{{\Q}}^{\perp}/\bar{J}_{{\Q}}$. 
This consists of Eichler transvections of $L(I)_{{\Q}}$ with respect to $\bar{J}_{{\Q}}$, 
hence isomorphic to 
$(\bar{J}_{{\Q}}^{\perp}/\bar{J}_{{\Q}})\otimes \bar{J}_{{\Q}} \simeq L(J)_{{\Q}}\otimes (J/I)_{{\Q}}$. 
In this way we obtain the exact sequence 
\begin{equation}\label{eqn: split WJQ auxiliary I} 
0 \to \bar{J}_{{\Q}}^{\perp}\cap L(I)_{{\Q}} \to {\WJQ} \to L(J)_{{\Q}}\otimes (J/I)_{{\Q}} \to 0. 
\end{equation}

We choose lifts 
$L(J)_{{\Q}}\hookrightarrow J_{{\Q}}^{\perp}$ and $(J/I)_{{\Q}}\hookrightarrow J_{{\Q}}$. 
This induces a section of \eqref{eqn: split WJQ auxiliary I}  
which consists of the Eichler transvections $E_{w}$ of $L_{{\Q}}$ with 
$w\in L(J)_{{\Q}}\otimes (J/I)_{{\Q}}$. 
Together with the splitting 
$\bar{J}_{{\Q}}^{\perp}\cap L(I)_{{\Q}} \simeq \bar{J}_{{\Q}} \oplus (L(J)_{{\Q}}\otimes I_{{\Q}})$,  
we obtain a bijection 
\begin{equation*} 
{\WJQ} \simeq \bar{J}_{{\Q}}\times (L(J)_{{\Q}}\otimes I_{{\Q}}) \times (L(J)_{{\Q}}\otimes (J/I)_{{\Q}}). 
\end{equation*}
This gives an isomorphism with the Heisenberg group. 
\end{proof}

Note that ${\UJQ}$ is not just the center of ${\WJQ}$, 
but also the center of ${\GJQ}$. 
This is the reason we put the determinant $1$ condition in the definition of ${\GJQ}$. 

The action of ${\GJQ}$ on the Siegel domain model can be described through the filtration 
${\UJQ}\subset {\WJQ}\subset {\GJQ}$. 
By definition ${\UJQ}$ acts on ${\VJ}\subset {\proj}(J_{{\C}}^{\perp})^{\vee}$ trivially 
and ${\WJQ}$ acts on ${\HJ}\subset {\proj}J_{{\C}}^{\vee}$ trivially. 
We let $U(J)_{{\C}}\subset {\rm O}(L_{{\C}})$ be the group of 
Eichler transvections $E_{l\otimes l'}$ with $l, l'\in J_{{\C}}$. 
Then $U(J)_{{\C}}\simeq \wedge^{2}J_{{\C}}$ acts on ${\D}(J)$ leaving ${\VJ}$ fixed. 
One observes that 
\begin{itemize}
\item ${\D}(J)\to {\VJ}$ is a principal $U(J)_{{\C}}$-bundle. 
\item ${\VJQ}\simeq L(J)_{{\Q}}\otimes J_{{\Q}}$ acts on ${\VJ}\to {\HJ}$ by 
relative translation through  
${\VJ}\simeq  L(J)_{{\C}} \otimes (\underline{J_{{\C}}}/F^{1,0})$. 
\item If we choose a splitting for $J_{{\Q}}$, 
which induces a splitting of \eqref{eqn: WJQ sequence} and 
the above isomorphism for ${\VJ}$, 
then the section of ${\rm O}(L(J)_{{\Q}})\times {\rm SL}(J_{{\Q}})$ acts on ${\VJ}\to {\HJ}$ by 
the equivariant action of ${\rm SL}(J_{{\Q}})$ on $\underline{J_{{\C}}}/F^{1,0}$ 
and the linear action of ${\rm O}(L(J)_{{\Q}})$ on $L(J)_{{\C}}$. 
\end{itemize}

\subsection{Stabilizer over ${\Z}$}\label{ssec: stab Z 1dim cusp}

Let ${\G}$ be a finite-index subgroup of ${\OL}$. 
We set  
\begin{equation*}
{\GJZ} = {\GJQ}\cap {\G}, \quad 
{\WJZ} = {\WJQ} \cap {\G}, \quad 
{\UJZ} = {\UJQ} \cap {\G}, 
\end{equation*}
and consider the quotients  
\begin{equation*}
{\GJZbar} = {\GJZ}/{\UJZ}, \quad 
{\VJZ} = {\WJZ}/{\UJZ}, \quad 
\Gamma_{J} = {\GJZ}/{\WJZ},  
\end{equation*}
\begin{equation*}
{\GJQbar} = {\GJQ}/{\UJZ}, \; \;  
{\WJQZ} = {\WJQ}/{\UJZ}, \; \;  
{\UJQZ} = {\UJQ}/{\UJZ}.  
\end{equation*}
By definition we have the canonical exact sequences 
\begin{equation}\label{eqn: GIZbar sequence}
0 \to {\VJZ} \to {\GJZbar} \to \Gamma_{J} \to 1, 
\end{equation}
\begin{equation}\label{eqn: GJQbar sequence}
0 \to {\WJQZ} \to {\GJQbar} \to {\rm O}(L(J)_{{\Q}})\times {\rm SL}(J_{{\Q}}) \to 1,   
\end{equation} 
\begin{equation*}
0 \to {\UJQZ} \to {\WJQZ} \to {\VJQ} \to 0.  
\end{equation*}
Then \eqref{eqn: GIZbar sequence} is canonically embedded in \eqref{eqn: GJQbar sequence}. 
We have ${\VJZ}\cap {\UJQZ} = \{ 0 \}$ as subgroups of ${\WJQZ}$  
because the projection ${\VJZ}\to {\VJQ}$ is injective. 
Note that ${\UJQZ}$ is the group of torsion points of the $1$-dimensional torus $T(J)=U(J)_{{\C}}/{\UJZ}$. 
The sequence \eqref{eqn: GJQbar sequence} splits (non-canonically) 
if we choose a splitting for $J_{{\Q}}$, 
though this does not mean that \eqref{eqn: GIZbar sequence} splits.

\subsection{Partial toroidal compactification}\label{ssec: toroidal cpt 1dim cusp}

We denote ${\TJ}={\D}(J)/{\UJZ}$ and ${\XJ}={\D}/{\UJZ}$. 
By the description so far, 
${\TJ}\to {\VJ}$ is a principal $T(J)$-bundle 
acted on equivariantly by 
${\GIZbar}<{\GIQbar}$, 
in which ${\XJ}$ is a punctured disc bundle. 
The action of ${\GIZbar}$ on ${\VJ}\to {\HJ}$ is described as follows. 
We choose a splitting for $J_{{\Q}}$ to give a zero section of ${\VJ}$  
and a splitting of ${\GJQbar}$ in \eqref{eqn: GJQbar sequence}. 
We express an element $\gamma$ of ${\GJZbar}\subset {\GJQbar}$ as 
$\gamma=(\gamma_{1}, \gamma_{2}, \alpha)$ 
accordingly, where 
$\gamma_{1}\in {\rm O}(L(J))$, $\gamma_{2}\in {\rm SL}(J)$ and $\alpha\in {\WJQZ}$. 
Then $\gamma$ acts on ${\VJ}\simeq L(J)_{{\C}}\otimes (\underline{J_{{\C}}}/F^{1,0})$ 
by the equivariant action by $(\gamma_{1}, \gamma_{2})$ and 
the translation by $[\alpha]\in {\VJQ}\simeq L(J)_{{\Q}}\otimes J_{{\Q}}$. 
Thus ${\VJ}/{\VJZ}$ is a fibration of abelian varieties over ${\HJ}$ 
isogenous to the self fiber product of the universal elliptic curve.  
The group $\Gamma_{J}$ acts on ${\VJ}/{\VJZ}$ by the equivariant action plus some possible translation. 

Now let $\overline{T(J)}\simeq {\C}$ be the canonical compactification of the torus $T(J)$. 
We take the relative torus embedding 
${\TJcpt}={\TJ}\times_{T(J)}\overline{T(J)}$, 
and let ${\XJcpt}$ be the interior of the closure of ${\XJ}$ in ${\TJcpt}$. 
This is the partial compactification of ${\XJ}$ over the $1$-dimensional cusp $J$. 
Note that no choice of fan is required: this is canonical. 
The boundary divisor of ${\XJcpt}$ is canonically isomorphic to ${\VJ}$.  

The relationship with a partial compactification over an adjacent $0$-dimensional cusp 
$I \subset J$ is as follows. 
Recall that 
$\bar{J}=(J/I)\otimes I \simeq \wedge^{2}J$ is an isotropic sublattice of $L(I)$, 
oriented by the orientation of $J$. 
The ray $\sigma_{J}=(\bar{J}_{{\R}})_{\geq 0}$ is in the closure of the positive cone, 
and it is contained in any ${\GIZ}$-admissible fan $\Sigma$. 
The torus embedding 
$T(I)\hookrightarrow T(I)^{\sigma_{J}}$ defined by $\sigma_{J}$ 
is a Zariski open set of $T(I)^{\Sigma}$. 
By \eqref{eqn: UJQ in UIQ} we have 
$U(J)_{{\R}}=\bar{J}_{{\R}}\subset U(I)_{{\R}}$ and 
\begin{equation}\label{eqn: UJZ in UIZ}
{\UJZ} = \bar{J}_{{\R}} \cap {\UIZ} = {\R}\sigma_{J} \cap {\UIZ}. 
\end{equation}
Therefore the inclusion ${\D}(J)\subset {\DI}$ induces the etale map 
\begin{equation}\label{eqn: glue}
{\TJcpt}\to T(I)^{\sigma_{J}}\subset T(I)^{\Sigma}, 
\end{equation} 
which maps the boundary divisor of ${\TJcpt}$ 
to the unique boundary divisor of $T(I)^{\sigma_{J}}$.  
It may be worth noting that ${\UIZ}\subset {\GJZ}$.


\section{Irregular $1$-dimensional cusps}\label{sec: irr 1dim cusp}

In this section we define and study irregular $1$-dimensional cusps. 
For simplicity we assume $b\geq 3$ so that $L(J)\ne \{ 0 \}$. 
Let ${\G}$ be a finite-index subgroup of ${\OL}$ and 
$J$ be a rank $2$ primitive isotropic sublattice of $L$. 
We keep the notation from \S \ref{sec: 1dim cusp}. 
Irregularity of the $1$-dimensional cusp $J$ can be characterized as follows. 

\begin{proposition}
The following conditions are equivalent. 
\begin{enumerate}
\item ${\UJZ}\ne {\UJZZ}$ where ${\UJZZ}={\UJQ}\cap {\GG}$. 
\item $-{\id}\not\in {\G}$ and $-E_{w}\in {\GJZ}$ for some $w\in \wedge^{2}J_{{\Q}}$. 
\item $-{\id}\not\in {\G}$ and ${\GJZbar}$ contains an element $\gamma$ of finite order 
whose image in ${\rm O}(L(J))\times {\rm SL}(J)$ is $(-{\id}_{L(J)}, -{\id}_{J})$. 
\item ${\GJZbar}$ contains an element $\gamma$ which acts trivially on ${\VJ}$ but nontrivially on ${\XJ}$. 
\end{enumerate}
When these hold, the element $\gamma$ of ${\GJZbar}$ in (3), (4) is given by $-E_{w}$ in (2), 
has order $2$, and is unique. 
\end{proposition}

\begin{definition}
We say that the $1$-dimensional cusp $J$ is \textit{irregular} 
when these properties hold, and \textit{regular} otherwise. 
\end{definition}

\begin{proof}
The equivalence $(1) \Leftrightarrow (2)$ is similar to $(1) \Leftrightarrow (2)$ in 
Proposition \ref{prop: characterize 0dim irregular}. 
The quotient ${\UJZZ}/{\UJZ}\simeq {\Z}/2$ is generated by $E_{w}$ in (2). 

$(2) \Rightarrow (4)$: 
Since $E_{w}$ for $w\in \wedge^{2}J_{{\Q}}$ acts trivially on ${\VJ}$, so does $-E_{w}$. 

$(2) \Rightarrow (3)$: 
The element $\gamma=[-E_{w}]$ of ${\GJZbar}$ is of order $2$ and acts on $J$, $L(J)$ by $-1$. 

$(3) \Rightarrow (4)$: 
By the description of the ${\GJZbar}$-action on ${\VJ}$ in \S \ref{ssec: toroidal cpt 1dim cusp}, 
we find that the element $\gamma$ of (3) acts on ${\VJ}$ by some translation. 
Since $\gamma$ is of finite order by assumption, 
this translation must be trivial. 

$(4) \Rightarrow (2), (3)$: 
Suppose that $\gamma \in {\GJZbar}$ acts trivially on ${\VJ}$ but nontrivially on ${\XJ}$. 
We take a splitting of ${\GJQbar}$ and express $\gamma$ as 
$(\gamma_{1}, \gamma_{2}, \alpha)$ as in \S \ref{ssec: toroidal cpt 1dim cusp}. 
Since $\gamma$ acts on $\mathbb{H}_{J}$ trivially, we must have 
$\gamma_{2}={\id}_{J}$ or $-{\id}_{J}$. 
Then, since $\gamma$ acts on ${\VJ}\simeq L(J)_{{\C}}\otimes (\underline{J_{{\C}}}/F^{1,0})$ trivially, 
we see that  
$(\gamma_{1}, \gamma_{2}) = ({\id}_{L(J)}, {\id}_{J})$ or $(-{\id}_{L(J)}, -{\id}_{J})$, 
and the image of $\alpha \in {\WJQZ}$ in ${\VJQ}$ must be $0$, 
namely $\alpha\in {\UJQZ}$. 
The case $(\gamma_{1}, \gamma_{2}) = ({\id}_{L(J)}, {\id}_{J})$ cannot occur 
because then 
$\gamma\in {\UJQZ}\cap {\VJZ} = \{ 0 \}$. 
Therefore 
$\gamma=(-{\id}_{L(J)}, -{\id}_{J}, E_{w})$ 
for some $w\in \wedge^{2}J_{{\Q}}$. 
Since 
$-{\id}_{L}=(-{\id}_{L(J)}, -{\id}_{J}, 0)$ 
with respect to this (and any) splitting, 
we find that 
$\gamma=-E_{w}$. 
Thus $-E_{w}\in {\G}$. 
Finally, we have $-{\id}\not\in {\G}$, 
for otherwise $E_{w}=-\gamma$ would be contained in ${\UJZ}$, which in turn implies that  
$\gamma$ acts trivially on ${\XJ}$. 
\end{proof}

As in the case of $0$-dimensional cusps, 
$U(J)_{{\Z}}'$ is the projection image of 
$U(J)_{{\Z}}^{\star}=(\{ \pm{\id} \} \cdot {\UJQ})\cap {\G}$ 
in ${\UJQ}$, and we have 
$U(J)_{{\Z}}^{\star}/{\UJZ}=\langle -E_{w} \rangle$ 
when $J$ is irregular. 

Since the boundary divisor of ${\XJcpt}$ is naturally isomorphic to ${\VJ}$, 
the condition $(4)$ can be restated as follows. 

\begin{corollary}\label{cor: irregular 1dim ramify}
A $1$-dimensional cusp $J$ is irregular if and only if 
${\XJcpt}\to {\XJcpt}/{\GJZbar}$ is ramified along 
the boundary divisor of ${\XJcpt}$. 
In that case, the ramification index is $2$, and 
the unique nontrivial element of ${\GJZbar}$ fixing the boundary divisor is given by $-E_{w}$. 
\end{corollary}

By the condition (1), irregularity of a $1$-dimensional cusp reduces to 
that of an adjacent $0$-dimensional cusp as follows. 

\begin{proposition}\label{prop: 1dim reduce to 0dim}
Let $I\subset J$ be a rank $1$ primitive sublattice and 
$\sigma_{J}\subset U(I)_{{\R}}$ be the isotropic ray corresponding to $J$. 
Then $J$ is irregular if and only if 
$I$ is irregular and $\sigma_{J}$ is an irregular ray. 
\end{proposition}

\begin{proof}
Recall that $\sigma_{J}$ is called irregular when 
${\R}\sigma_{J}\cap {\UIZ} \ne {\R}\sigma_{J} \cap {\UIZZ}$. 
By \eqref{eqn: UJZ in UIZ} we have 
${\R}\sigma_{J}\cap {\UIZ} = {\UJZ}$, 
and similarly 
${\R}\sigma_{J}\cap {\UIZZ} = {\UJZZ}$. 
This proves our assertion.  
\end{proof}

\begin{corollary}\label{cor: 1dim reduced to 0dim}
If ${\G}$ has no irregular $0$-dimensional cusp, 
it has no irregular $1$-dimensional cusp. 
\end{corollary}


\section{Toroidal compactification}\label{sec: toroidal cpt}

In this section we study singularities and ramification divisors in the boundary  
of a toroidal compactification of the modular variety. 
These are studied in \cite{GHS07}, \cite{Ma} under the condition $-{\id}\in {\G}$, 
and we explain what modification is necessary in the general case, 
especially at the irregular cusps. 

Let $L$ be a lattice of signature $(2, b)$ and 
${\G}$ be a subgroup of ${\OL}$ of finite index. 
The input data for constructing a toroidal compactification of 
${\FG}={\G}\backslash {\D}$ is 
a collection $\Sigma=(\Sigma_{I})_{I}$ of ${\GIZ}$-admissible fans (\S \ref{ssec: toroidal cpt 0dim}), 
one for each ${\G}$-equivalence class of rank $1$ primitive isotropic sublattices $I$ of $L$. 
No choice is required for $1$-dimensional cusps. 
Thus $\Sigma$ is a finite collection of independent fans. 

The toroidal compactification associated to $\Sigma$ 
is defined as (\cite{AMRT} p.163) 
\begin{equation*}
{\FGcpt} = \left( {\D} \sqcup \bigsqcup_{I} \mathcal{X}(I)^{\Sigma_{I}} 
\sqcup \bigsqcup_{J}{\XJcpt} \right) / \sim, 
\end{equation*}
where 
$I$ (resp.~$J$) ranges over all primitive isotropic sublattices of $L$ of rank $1$ (resp.~$2$), 
and $\sim$ is the equivalence relation generated by 
the following: 
\begin{itemize}
\item Action of $\gamma\in {\G}$ giving ${\D}\to {\D}$, 
$\mathcal{X}(I)^{\Sigma_{I}}\to \mathcal{X}(\gamma I)^{\Sigma_{\gamma I}}$ and 
${\XJcpt} \to \overline{\mathcal{X}(\gamma J)}$. 
\item The natural maps ${\D}\to \mathcal{X}(I)^{\Sigma_{I}}$ and ${\D}\to {\XJcpt}$. 
\item The etale gluing maps ${\XJcpt}\to \mathcal{X}(I)^{\Sigma_{I}}$ for $I\subset J$ 
given by \eqref{eqn: glue}. 
\end{itemize}

\begin{theorem}[\cite{AMRT}]\label{thm: AMRT}
The space ${\FGcpt}$ is a compact Moishezon space containing ${\FG}$ as a Zariski open set, 
and we have a morphism from ${\FGcpt}$ to the Baily-Borel compactification of ${\FG}$. 
For each cusp $I$, $J$, the natural map 
\begin{equation*}
\mathcal{X}(I)^{\Sigma_{I}}/{\GIZbar} \to {\FGcpt}, \qquad 
{\XJcpt}/{\GJZbar} \to {\FGcpt}
\end{equation*}
is locally isomorphic in an open neighborhood of boundary points lying over that cusp. 
\end{theorem}

Perhaps a word might be in order because, strictly speaking, 
the theory of \cite{AMRT} is applied to the image of ${\G}$ in ${\rm O}^{+}(L_{{\R}})/\pm {\id}$, 
which is ${\GG}/\pm{\id}$, rather than ${\G}$ itself. 
Then ${\UIZ}$ should be replaced by ${\UIZZ}$,  
${\XI}={\D}/{\UIZ}$ by ${\XII}={\D}/{\UIZZ}$, 
${\GIZ}$ by $\Gamma'(I)_{{\Z}}=\langle {\GIZ}, -{\id} \rangle / \pm {\id}$, 
and similarly for $1$-dimensional cusps $J$. 
But since ${\XII}={\XI}$ or 
${\XII}={\XI}/\langle -E_{w} \rangle$ with $-E_{w}\in {\GIZ}$ 
(and similarly for $J$), we have naturally 
\begin{equation*}
\left( {\D} \sqcup \bigsqcup_{I} \mathcal{X}(I)^{\Sigma_{I}} \sqcup \bigsqcup_{J}{\XJcpt} \right) / \sim 
\; \; = \; \; 
\left( {\D} \sqcup \bigsqcup_{I} ({\XII})^{\Sigma_{I}} \sqcup \bigsqcup_{J}\overline{\mathcal{X}(J)'} \right) / \sim', 
\end{equation*}
where $\sim'$ is the equivalence relation similar to $\sim$. 
The last statement of Theorem \ref{thm: AMRT} (\cite{AMRT} p.175) is justified because we have 
\begin{equation*}
\mathcal{X}(I)^{\Sigma_{I}}/{\GIZbar} = 
({\XII})^{\Sigma_{I}} / (\Gamma'(I)_{{\Z}}/{\UIZZ}) 
\end{equation*}
(see also \eqref{eqn: /GIZ vs /GIZZ}), 
and similarly for $J$.

The reason we prefer to work with ${\UIZ}$ rather than ${\UIZZ}$ 
is that Fourier expansion of ${\G}$-modular forms of arbitrary weight 
can be done with ${\UIZ}$ 
(see \S \ref{sec: cusp form criterion}). 

If $D(\sigma)\subset \mathcal{X}(I)^{\Sigma_{I}}$ is 
the boundary divisor corresponding to a ray $\sigma\in \Sigma_{I}$, 
general points of $D(\sigma)$ lie over the $I$-cusp if and only if $\sigma$ is positive-definite. 
When $\sigma=\sigma_{J}$ is isotropic corresponding to a $1$-dimensional cusp $J\supset I$, 
$D(\sigma_{J})$ is glued with the boundary divisor of ${\XJcpt}$, 
and its general points lie over the $J$-cusp. 
By combining the last statement of Theorem \ref{thm: AMRT} with 
Corollaries \ref{cor: irregular 0dim ramify} and \ref{cor: irregular 1dim ramify}, 
we obtain the following. 

\begin{proposition}\label{prop: boundary ramification divisor}
(1) The projection $\mathcal{X}(I)^{\Sigma_{I}}\to {\FGcpt}$ is ramified along 
irregular boundary divisors of $\mathcal{X}(I)^{\Sigma_{I}}$ with ramification index $2$, 
and not ramified along other boundary divisors. 
If we take quotient by $U(I)_{{\Z}}^{\star}/{\UIZ}$, 
then $({\D}/U(I)_{{\Z}}^{\star})^{\Sigma_{I}}\to {\FGcpt}$ is not ramified along the boundary divisors. 

(2) The projection ${\XJcpt}\to {\FGcpt}$ is ramified along the unique boundary divisor 
(with index $2$) if and only if $J$ is irregular. 
If we take quotient by $U(J)_{{\Z}}^{\star}/{\UJZ}$, 
then $\overline{{\D}/U(J)_{{\Z}}^{\star}}\to {\FGcpt}$ is not ramified along the boundary divisor. 
\end{proposition}

\begin{proof}
What remains is to show that (1) is still true even when a ray $\sigma=\sigma_{J}$ is isotropic. 
Since the map ${\XJcpt}\to {\FGcpt}$ in (2) factorizes as 
${\XJcpt}\to \mathcal{X}(I)^{\Sigma_{I}} \to {\FGcpt}$ 
and the gluing map ${\XJcpt}\to \mathcal{X}(I)^{\Sigma_{I}}$ is etale, 
our assertion for $\mathcal{X}(I)^{\Sigma_{I}} \to {\FGcpt}$ follows from (2) 
and Proposition \ref{prop: 1dim reduce to 0dim}. 
\end{proof}

When ${\G}$ contains $-{\id}$, 
Proposition \ref{prop: boundary ramification divisor} is proved in \cite{GHS07}, \cite{Ma}. 
In that case, we have no irregular cusp, so no ramification divisor in the boundary. 


\begin{remark}
It appears that in some literatures, 
the ``no ramification boundary divisor'' property is used to claim that 
$\mathcal{F}(\Gamma')^{\Sigma}\to {\FGcpt}$ is not ramified along the boundary divisors 
for neat subgroups $\Gamma' < {\G}$. 
This seems not true already in the case of modular curves: for example, 
$\Gamma(N) < {\rm SL}_{2}({\Z})$. 
The point is that 
$U(I)_{{\Z},{\G}}={\UIQ}\cap {\G}$ depends on ${\G}$, 
so $U(I)_{{\Z},{\G}'}={\UIQ}\cap {\G}'$ is in general smaller than $U(I)_{{\Z},{\G}}$. 
If $\sigma$ is a ray in $\Sigma_{I}$, assumed regular for simplicity, we have ramification index 
\begin{equation*}
[{\R}\sigma\cap U(I)_{{\Z},{\G}} : {\R}\sigma\cap U(I)_{{\Z},{\G}'}] 
\end{equation*}
at the corresponding boundary divisor.  
It seems that so far, all argument using the above claim can be avoided: 
see the proof of Theorem \ref{thm: low slope}. 
\end{remark}

Next we study singularities. 
A fan $\Sigma_{I}=(\sigma_{\alpha})$ is called \textit{basic} with respect to 
a lattice $\Lambda \subset {\UIQ}$ if 
each cone $\sigma_{\alpha}$ is generated by a part of a basis of $\Lambda$. 
The singularity theorem (\cite{GHS07}, \cite{Ma}) is still true, 
if we require the fan $\Sigma_{I}$ to be basic with respect to ${\UIZZ}$, rather than ${\UIZ}$. 

\begin{proposition}[cf.~\cite{GHS07}, \cite{Ma}]\label{thm: singularity}
(1) We choose the fans $\Sigma=(\Sigma_{I})$ so that 
each $\Sigma_{I}$ is basic with respect to ${\UIZZ}$. 
Then ${\FGcpt}$ has canonical singularities 
at the boundary points lying over the $0$-dimensional cusps. 

(2) When $b\geq 9$, ${\FGcpt}$ has canonical singularities 
at the boundary points lying over the $1$-dimensional cusps. 
\end{proposition}

\begin{proof}
When ${\G}$ contains $-{\id}$, 
this is proved in \cite{GHS07}, \cite{Ma} for $0$-dimensional cusps, 
and in \cite{GHS07} for $1$-dimensional cusps. 
We show that the general case is reduced to this case. 
We consider $0$-dimensional cusps. 
The case of $1$-dimensional cusps is similar. 
It suffices to show that 
${\XI}^{\Sigma_{I}}/{\GIZbar}$ 
has canonical singularities. 

Let $\Gamma' = {\GG}$ and 
$\Gamma'(I)_{{\Z}}=\Gamma' \cap {\GIQ}$. 
Then ${\UIZZ}={\UIQ}\cap \Gamma'$ and $\Gamma'(I)_{{\Z}} = \langle {\GIZ}, -{\id} \rangle$. 
Since the fan $\Sigma_{I}$ is also rational with respect to ${\UIZZ}$, 
it defines a toroidal embedding 
$({\D}/{\UIZZ})^{\Sigma_{I}}$ of ${\D}/{\UIZZ}$. 
This is the quotient of $({\D}/{\UIZ})^{\Sigma_{I}}$ 
by the translation by ${\UIZZ}/{\UIZ}$ 
(which is nontrivial exactly when $I$ is irregular). 
Since ${\UIZZ}/{\UIZ}\subset \Gamma'(I)_{{\Z}}/{\UIZ}$,
we have 
\begin{eqnarray}\label{eqn: /GIZ vs /GIZZ}
({\D}/{\UIZ})^{\Sigma_{I}}/{\GIZbar} 
& = & 
({\D}/{\UIZ})^{\Sigma_{I}}/(\Gamma'(I)_{{\Z}}/{\UIZ}) \\ \nonumber  
& \simeq & 
({\D}/{\UIZZ})^{\Sigma_{I}}/(\Gamma'(I)_{{\Z}}/{\UIZZ}). 
\end{eqnarray}
Since $\Sigma_{I}$ is basic with respect to ${\UIZZ}$ and $-{\id}\in \Gamma'$, 
we can apply the result of \cite{Ma} to the last quotient 
to see that this has canonical singularities. 
\end{proof}


\section{Modular forms and pluricanonical forms}\label{sec: cusp form criterion}

Let $L$ be a lattice of signature $(2, b)$ 
and ${\G}$ be a subgroup of ${\OL}$ of finite index. 
For simplicity we assume $b\geq 3$.  
In this section we compare vanishing order of cusp forms and pluricanonical forms, 
and explain how the low weight cusp form trick of Gritsenko-Hulek-Sankaran \cite{GHS07} 
is modified at irregular boundary divisors. 
We take this occasion to generalize ``low weight'' to ``low slope'', 
for possible future use.

\subsection{Modular forms}\label{ssec: modular form}

Let $\mathcal{L}=\mathcal{O}_{{\proj}L_{{\C}}}(-1)|_{{\D}}$ 
be the restriction of the tautological line bundle to ${\D}\subset {\proj}L_{{\C}}$. 
Let $\chi$ be a character of ${\G}$. 
By our assumption $b\geq 3$, 
$\chi({\G})\subset {\C}^{\times}$ is finite (\cite{Margu}). 
We assume that $\chi|_{{\UIZ}}\equiv 1$ for every $0$-dimensional cusp $I$. 
This holds, e.g., for $\chi=1$ and $\chi=\det$. 
A ${\G}$-invariant section of the ${\G}$-linearized line bundle 
$\mathcal{L}^{\otimes k}\otimes \chi$ over ${\D}$ 
is called a \textit{modular form} of weight $k$ and character $\chi$ with respect to ${\G}$. 

Let $I$ be a rank $1$ primitive isotropic sublattice of $L$. 
We choose a generator $l_{I}$ of $I$. 
This defines a frame $s_{I}$ of $\mathcal{L}$ determined by the condition 
$(s_{I}([\omega]), l_{I})=1$, 
where we view 
$s_{I}([\omega])\in \mathcal{L}_{[\omega]}={\C}\omega \subset L_{{\C}}$. 
The factor of automorphy with respect to $s_{I}$ is given by 
\begin{equation*}
j(\gamma, [\omega]) = 
\frac{(\gamma\omega, l_{I})}{(\omega, l_{I})} = 
\frac{(\omega, \gamma^{-1}l_{I})}{(\omega, l_{I})} , \qquad 
\gamma\in {\G}, \; [\omega]\in {\D}. 
\end{equation*}
Let $1_{\chi}$ be a nonzero vector in the representation line of $\chi$. 
Then $s_{I}^{\otimes k}\otimes 1_{\chi}$ is a frame of the line bundle 
$\mathcal{L}^{\otimes k}\otimes \chi$, 
via which modular forms $F=f s_{I}^{\otimes k}\otimes 1_{\chi}$ 
of weight $k$ and character $\chi$ are identified with 
holomorphic functions $f$ on ${\D}$ satisfying 
\begin{equation*}
f(\gamma [\omega]) = \chi(\gamma)j(\gamma, [\omega])^{k}f([\omega]), 
\quad \gamma \in {\G}, \; [\omega]\in {\D}. 
\end{equation*}

Since $s_{I}^{\otimes k}\otimes 1_{\chi}$ is invariant under ${\UIZ}$ by our assumption, 
$f$ is ${\UIZ}$-invariant, hence descends to a function on ${\D}/{\UIZ}$. 
By the tube domain realization ${\D}\to {\D}_{I}\subset U(I)_{{\C}}$ 
(after a choice of $I'\subset L$ with $(I, I')\not\equiv 0$), 
$f$ is identified with a function on ${\D}_{I}$ invariant under translation by the lattice ${\UIZ}$. 
Then it admits a Fourier expansion 
\begin{equation}\label{eqn: Fourier expansion}
f(Z) = \sum_{l\in U(I)_{{\Z}}^{\vee}} a(l) q^{l}, \quad 
q^{l}={\rm exp}(2\pi i(l, Z)), \; Z\in \mathcal{D}_{I}. 
\end{equation}
By the Koecher principle, we have $a(l)\ne 0$ only when $l\in \overline{{\CI}}$. 
The modular form $F$ is called a \textit{cusp form} 
if $a(l)=0$ for every $l\in U(I)_{{\Z}}^{\vee}$ with $(l, l)=0$ at every rank $1$ primitive isotropic sublattice $I$ of $L$. 
($a(0)$ is the value of $f$ at the $0$-dimensional cusp for $I$, and 
$\sum_{\sigma\cap U(I)_{{\Z}}^{\vee}}a(l)q^{l}$ for an isotropic ray $\sigma=\sigma_{J}$ 
gives the restriction of $f$ to the $1$-dimensional cusp for $J\supset I$.)

Fourier expansion at an irregular cusp satisfies the following. 

\begin{lemma}\label{lem: irregular Fourier}
Suppose that $I$ is irregular and $-E_{w}\in {\GIZ}$. 
When the weight $k$ satisfies $\chi(-E_{w})=(-1)^{k+1}$, 
e.g., $k$ odd for $\chi=1$ or $k\not\equiv b$ mod $2$ for $\chi=\det$, 
then we have $a(l)=0$ for $l\in (U(I)_{{\Z}}')^{\vee}$. 
In particular, $a(0)=0$ in this case. 
When $\chi(-E_{w})=(-1)^{k}$, we have $a(l)=0$ for $l\not\in (U(I)_{{\Z}}')^{\vee}$. 
\end{lemma}

\begin{proof}
Since the factor of automorphy of $-E_{w}$ is $-1$, 
we find that $f(Z+w)=\chi(-E_{w})(-1)^{k}f(Z)$. 
On the other hand, we have 
$(l, w)\in {\Z}$ if $l\in (U(I)_{{\Z}}')^{\vee}$ and 
$(l, w)\in 1/2+{\Z}$ if $l\in U(I)_{{\Z}}^{\vee}-(U(I)_{{\Z}}')^{\vee}$. 
Therefore, if we substitute $Z\to Z+w$ into $q^{l}={\rm exp}(2\pi i(l, Z))$, 
then $q^{l}\to q^{l}$ if $l\in (U(I)_{{\Z}}')^{\vee}$ 
and $q^{l}\to -q^{l}$ if $l\in U(I)_{{\Z}}^{\vee}-(U(I)_{{\Z}}')^{\vee}$. 
This implies our assertion. 
\end{proof}

\subsection{Vanishing order}\label{ssec: vanishing order}

In this subsection we study vanishing order of modular forms along boundary divisors. 
We will define two types of vanishing order: 
$\nu_{\sigma}(F)$ and $\nu_{\sigma, geom}(F)$. 
$\nu_{\sigma}(F)$ is defined by Fourier expansion and is always an integer. 
On the other hand, $\nu_{\sigma, geom}(F)$ can be strictly half-integral, 
and measures the vanishing order at the level of ${\FGcpt}$. 

Let $I$ be a rank $1$ primitive isotropic sublattice of $L$. 
Let $\Sigma=\Sigma_{I}=(\sigma_{\alpha})$ be a ${\GIZ}$-admissible fan in $U(I)_{{\R}}$ 
and $\sigma$ be a ray in $\Sigma$. 
Let $w_{\sigma}$ be the generator of $\sigma \cap {\UIZ}$. 
Let 
$f(Z)=\sum_{l\in U(I)_{{\Z}}^{\vee}}a(l)q^{l}$ 
be the Fourier expansion of a ${\G}$-modular form 
$F=f s_{I}^{\otimes k}\otimes 1_{\chi}$ around $I$. 
We define the vanishing order of $F$ along $\sigma$ as 
\begin{equation*}\label{eqn: vanishing order}
\nu_{\sigma}(F) = \min \{ \: (l, w_{\sigma}) \: | \: l\in U(I)_{{\Z}}^{\vee}, \: a(l)\ne 0 \: \}. 
\end{equation*}
This is a nonnegative integer. 
Clearly $\nu_{\sigma}(F)$ depends on ${\UIZ}$ and hence on ${\G}$. 
If we shrink ${\G}$ without changing $F$ and $\sigma$, 
then $\nu_{\sigma}(F)$ will be multiplied in general. 

When $\sigma$ is positive-definite, 
we have $\nu_{\sigma}(F)>0$ if and only if $a(0)=0$, 
because $\sigma^{\perp}\cap \overline{{\CI}} = \{ 0 \}$. 
When $\sigma$ is isotropic, 
we have $\nu_{\sigma}(F)>0$ if and only if $a(l)=0$ for all $l\in \sigma\cap U(I)_{{\Z}}^{\vee}$, 
because $\sigma^{\perp}\cap \overline{{\CI}} = \sigma$. 
Thus, $F$ is a cusp form if and only if 
$\nu_{\sigma}(F)>0$ at every ray $\sigma$ at every $0$-dimensional cusp $I$.

The following criterion is trivial but perhaps might be sometimes useful in view of Theorem \ref{thm: low slope}.  
Compare with \cite{AMRT}, \cite{Fr} in related cases. 
 
\begin{corollary}
Assume the following holds: 
if $a(l)\ne 0$, then $(l, w)\geq r$ for every $w\in {\UIZ}\cap \overline{{\CI}}$. 
Then we have $\nu_{\sigma}(F)\geq r$ for every ray $\sigma\in \Sigma$.
\end{corollary}

\begin{proof}
Take $w$ to be the generator of $\sigma \cap {\UIZ}$. 
\end{proof}

When $\sigma$ is irregular, $\nu_{\sigma}(F)$ belongs to the following parity. 

\begin{proposition}\label{prop: irregular vanishing order}
Suppose that the ray $\sigma$ is irregular and $-E_{w}\in {\GIZ}$. 
Then $\nu_{\sigma}(F)$ is odd when $\chi(-E_{w})=(-1)^{k+1}$, 
and even when $\chi(-E_{w})=(-1)^{k}$. 
\end{proposition}

\begin{proof}
Let $w_{\sigma}$ be the generator of $\sigma\cap {\UIZ}$. 
Since ${\UIZZ}=\langle {\UIZ}, w_{\sigma}/2 \rangle$, 
a vector $l$ of $U(I)_{{\Z}}^{\vee}$ belongs to $(U(I)_{{\Z}}')^{\vee}$ 
if and only if $(l, w_{\sigma})$ is even. 
Then our assertion follows from Lemma \ref{lem: irregular Fourier}. 
\end{proof}

We also define the geometric vanishing order of $F$ along $\sigma$ as 
\begin{equation*}\label{eqn: geom vanishing order}
\nu_{\sigma, geom}(F) = 
\begin{cases} 
\nu_{\sigma}(F) & \sigma :  \textrm{regular} \\ 
\frac{1}{2}\nu_{\sigma}(F) & \sigma :  \textrm{irregular} 
\end{cases} 
\end{equation*}
If $w_{\sigma}'$ is the generator of $\sigma \cap {\UIZZ}$, 
we can write uniformly as 
\begin{equation}\label{eqn: vanishing order}
\nu_{\sigma, geom}(F) = \min \{ \: (l, w_{\sigma}') \: | \: l\in U(I)_{{\Z}}^{\vee}, \: a(l)\ne 0 \: \}. 
\end{equation}
Note that $\nu_{\sigma, geom}(F)$ is in $1/2+{\Z}$ when $\sigma$ is irregular and 
the weight $k$ satisfies $\chi(-E_{w})=(-1)^{k+1}$ so that $\nu_{\sigma}(F)$ is odd. 

Geometric interpretation of $\nu_{\sigma}(F)$ is as follows. 
Recall that the ray $\sigma$ corresponds to a boundary divisor $D(\sigma)$ of the partial compactification 
${\XIcpt}$ of ${\XI}={\D}/{\UIZ}$. 
The line bundle $\mathcal{L}^{\otimes k} \otimes \chi$ descends to a line bundle over ${\XI}$, 
again denoted by $\mathcal{L}^{\otimes k} \otimes \chi$. 
The point is that, since $s_{I}^{\otimes k}\otimes 1_{\chi}$ is ${\UIZ}$-invariant, 
it descends to a frame of $\mathcal{L}^{\otimes k} \otimes \chi$ over ${\XI}$,  
and we use this frame to extend $\mathcal{L}^{\otimes k} \otimes \chi$ 
to a line bundle over ${\XIcpt}$, still denoted by the same notation. 
Namely, $s_{I}^{\otimes k}\otimes 1_{\chi}$ extends to a frame of the extended line bundle by definition. 
The property $l \in \overline{{\CI}}=\overline{{\CI}}^{\vee}$ in the Fourier expansion  
implies that a modular form $F$ extends holomorphically over ${\XIcpt}$ 
as a section of $\mathcal{L}^{\otimes k}\otimes \chi$. 

\begin{proposition}\label{prop: geom interpret nu(F)}
$\nu_{\sigma}(F)$ is equal to the vanishing order of $F$ as a section of 
$\mathcal{L}^{\otimes k} \otimes \chi$ over ${\XIcpt}$ along the boundary divisor $D(\sigma)$. 
\end{proposition}

\begin{proof}
Recall that $\sigma$ defines a sub toroidal embedding 
$\mathcal{X}(I)^{\sigma}\subset {\XIcpt}$, 
the unique boundary divisor of which is a Zariski open set of $D(\sigma)$ 
and is the quotient torus (or its analytic open set) defined by the quotient lattice 
${\UIZ}/{\Z}w_{\sigma}$. 
The character group of this boundary torus is $\sigma^{\perp}\cap U(I)_{{\Z}}^{\vee}$. 
We choose a vector $l_{\sigma} \in U(I)_{{\Z}}^{\vee}$ 
such that $(l_{\sigma}, w_{\sigma})=1$ and put 
$q=q^{l_{\sigma}}$, which is a character of $T(I)$. 
Then $q$ extends holomorphically over $\mathcal{X}(I)^{\sigma}$ 
with $D(\sigma)=(q=0)$. 
The Fourier expansion \eqref{eqn: Fourier expansion} can be arranged as 
$f=\sum_{m\geq 0}\varphi_{m}q^{m}$ where 
\begin{equation*}
\varphi_{m} = 
\sum_{l\in \sigma^{\perp}\cap U(I)_{{\Z}}^{\vee}} 
a(l+ml_{\sigma})q^{l}.  
\end{equation*}
This is a Taylor expansion of $f$ along the divisor $D(\sigma)$. 
Since 
$(l+ml_{\sigma}, w_{\sigma})=m$ for 
$l\in \sigma^{\perp}\cap U(I)_{{\Z}}^{\vee}$, 
we find that 
\begin{equation*}
\nu_{\sigma}(F) = \min \{ \: m \: | \: \varphi_{m}\not\equiv 0 \: \}. 
\end{equation*}
This proves our assertion. 
\end{proof}

We can also give a geometric interpretation of $\nu_{\sigma, geom}(F)$ when 
\begin{equation}\label{eqn: UIZstar invariance}
s_{I}^{\otimes k}\otimes 1_{\chi} \; \textrm{is invariant under} \; 
U(I)_{{\Z}}^{\star} = (\{ \pm{\id} \} \cdot {\UIQ})\cap {\G}. 
\end{equation}
This holds, 
e.g., when $k$ is even with $\chi=1$ and when $k\equiv b$ mod $2$ with $\chi=\det$. 
Recall that ${\UIZZ}$ is the image of $U(I)_{{\Z}}^{\star}$ in ${\UIQ}$. 
Under the condition \eqref{eqn: UIZstar invariance}, 
the function $f(Z)$ on the tube domain $\mathcal{D}_{I}$ is invariant under translation by ${\UIZZ}$, 
so the index lattice in the Fourier expansion reduces to 
$({\UIZZ})^{\vee}\subset U(I)_{{\Z}}^{\vee}$. 
In other words, $a(l)=0$ if $l\not\in ({\UIZZ})^{\vee}$,  
so $\nu_{\sigma, geom}(F)$ is an integer. 
The frame $s_{I}^{\otimes k}\otimes 1_{\chi}$ descends to 
a frame of $\mathcal{L}^{\otimes k}\otimes \chi$ over 
\begin{equation*}
{\XII} = {\D}/U(I)_{{\Z}}^{\star} = {\D}/{\UIZZ}, 
\end{equation*}
using which we can extend $\mathcal{L}^{\otimes k}\otimes \chi$ 
to a line bundle over ${\XIIcpt}$. 
The ray $\sigma$ corresponds to a boundary divisor 
$D(\sigma)'$ of ${\XIIcpt}$. 
We have 
\begin{itemize}
\item $D(\sigma)'=D(\sigma)$ in ${\XIcpt}={\XIIcpt}$ 
when $I$ is regular. 
\item $D(\sigma)'\simeq D(\sigma)$ with ${\XIcpt}\to {\XIIcpt}$ doubly ramified along $D(\sigma)'$ 
when $\sigma$ is irregular. 
\item $D(\sigma)'$ is the quotient of $D(\sigma)$ by ${\UIZZ}/{\UIZ}\simeq {\Z}/2$ with 
${\XIcpt}\to {\XIIcpt}$ unramified along $D(\sigma)'$ 
when $I$ is irregular but $\sigma$ is regular. 
\end{itemize} 
Then we see, 
either from Proposition \ref{prop: geom interpret nu(F)} or by a similar argument, 
the following. 

\begin{proposition}\label{prop: geom interpret ord geom}
When \eqref{eqn: UIZstar invariance} holds, 
$\nu_{\sigma, geom}(F)$ is equal to the vanishing order of $F$ 
as a section of $\mathcal{L}^{\otimes k}\otimes \chi$ 
over ${\XIIcpt}$ along the boundary divisor $D(\sigma)'$. 
\end{proposition}

Vanishing order at a $1$-dimensional cusp $J$ 
is reduced to the case considered above. 
We choose a rank $1$ primitive sublattice $I\subset J$ and let 
$\sigma_{J}$ be the isotropic ray in $U(I)_{{\R}}$ corresponding to $J$. 
Then we define 
\begin{equation*}
\nu_{J}(F) = \nu_{\sigma_{J}}(F), \qquad 
\nu_{J, geom}(F) = \nu_{\sigma_{J}, geom}(F). 
\end{equation*}
The Taylor expansion $f=\sum_{m}\varphi_{m}q^{m}$ 
in this case is nothing but the Fourier-Jacobi expansion, 
and $\varphi_{m}$ is essentially the $m$-th Fourier-Jacobi coefficient. 
Thus $\nu_{J}(F)$ is the minimal degree of nonzero Fourier-Jacobi coefficients. 

We also have the following geometric interpretation of $\nu_{J}(F)$. 
We use the ${\UJZ}$-invariant frame $s_{I}^{\otimes k}\otimes 1_{\chi}$ 
to extend $\mathcal{L}^{\otimes k}\otimes \chi$ to a line bundle  
over ${\XJcpt}$. 
This is the pullback of the extended line bundle $\mathcal{L}^{\otimes k}\otimes \chi$ over ${\XIcpt}$ 
by the etale gluing map ${\XJcpt}\to {\XIcpt}$. 
This extension does not depend on the choice of $I$ up to isomorphism. 
Then $\nu_{J}(F)$ is the vanishing order of $F$ as a section of the extended line bundle 
$\mathcal{L}^{\otimes k}\otimes \chi$ over ${\XJcpt}$ along the boundary divisor. 
Similarly, when $s_{I}^{\otimes k}\otimes 1_{\chi}$ is invariant under 
$U(J)_{{\Z}}^{\star}$, 
$\nu_{J,geom}(F)$ equals to the vanishing order of $F$ 
along the boundary divisor of 
$\overline{{\D}/U(J)_{{\Z}}^{\star}}=\overline{{\D}/{\UJZZ}}$.

\subsection{Pluricanonical forms}\label{ssec: pluricanonical form}

In this subsection 
we compare vanishing order of modular forms and pluricanonical forms along the boundary divisors. 
Recall that we have a canonical isomorphism 
\begin{equation*}
\mathcal{L}^{\otimes b}\otimes \det \simeq K_{\mathcal{D}} 
\end{equation*} 
over ${\D}$, 
as a consequence of the isomorphism 
$K_{{\proj}L_{{\C}}}\simeq \mathcal{O}_{{\proj}L_{{\C}}}(-b-2)\otimes \det$ 
and the adjunction formula. 
Let $I$ be a rank $1$ primitive isotropic sublattice of $L$. 
The above isomorphism descends to 
$\mathcal{L}^{\otimes b}\otimes \det \simeq K_{{\XII}}$ 
over ${\XII}={\D}/U(I)_{{\Z}}^{\star}$. 
Both line bundles are extended over the partial compactification  
${\XIIcpt}$ in the respective manner: 
$\mathcal{L}^{\otimes b}\otimes \det$ is extended by the frame $s_{I}^{\otimes b}\otimes 1_{\det}$, 
while $K_{{\XII}}$ is extended to $K_{{\XIIcpt}}$. 

\begin{proposition}[cf.~\cite{Mu}]\label{prop: K vs L}
Over ${\XIIcpt}$ 
the above isomorphism extends to  
\begin{equation*}
\mathcal{L}^{\otimes b}\otimes \det \: \simeq \: 
K_{{\XIIcpt}} ( \sum_{\sigma}D(\sigma)' ), 
\end{equation*}
where $\sigma$ ranges over all rays in $\Sigma$ and 
$D(\sigma)'$ is the boundary divisor of ${\XIIcpt}$ corresponding to $\sigma$. 
\end{proposition}

\begin{proof}
By the isomorphism 
$\mathcal{L}^{\otimes b}\otimes \det \simeq K_{{\D}}$, 
the frame $s_{I}^{\otimes b}\otimes 1_{\det}$ of $\mathcal{L}^{\otimes b}\otimes \det$ 
corresponds to a flat canonical form $\omega_{I}$ on the tube domain 
${\D}_{I}\subset U(I)_{{\C}}$, 
because both extend over ${\D}(I)\simeq U(I)_{{\C}}$ and are $U(I)_{{\C}}$-invariant. 
Let $\sigma$ be a ray in $\Sigma$ and  
$w_{\sigma}'$ be the generator of $\sigma\cap {\UIZZ}$. 
We take a vector $l_{\sigma}\in ({\UIZZ})^{\vee}$ with $(l_{\sigma}, w_{\sigma}')=1$ 
and extend it to a basis of $({\UIZZ})^{\vee}$. 
This defines a coordinate $Z_{1}=(l_{\sigma}, \cdot), Z_{2}, \cdots , Z_{b}$ on $U(I)_{{\C}}$. 
We have $\omega_{I}=dZ_{1}\wedge \cdots \wedge dZ_{b}$ up to constant. 
Then $q=q^{l_{\sigma}}, Z_{2}, \cdots , Z_{b}$ define a local coordinate around 
a point of $D(\sigma)'\subset {\XIIcpt}$ with 
$D(\sigma)'=(q=0)$. 
Since we have 
\begin{equation*}
s_{I}^{\otimes b}\otimes 1_{\det} = 
dZ_{1} \wedge \cdots \wedge dZ_{b}  = 
\frac{dq}{q}\wedge dZ_{2} \wedge \cdots \wedge dZ_{b} 
\end{equation*}
around a point of $D(\sigma)'$,  
this proves our assertion. 
\end{proof}

This is the situation at a local chart for the boundary. 
We pass to the global situation. 

\begin{proposition}\label{prop: vanishing order relation}
Let $F$ be a modular form of weight $mb$ and character $\det^{m}$ with respect to ${\G}$ 
and $\omega_{F}$ be the corresponding rational $m$-canonical form on ${\FGcpt}$. 
Let $I$ be a $0$-dimensional cusp, $\sigma$ be a ray in $\Sigma_{I}$, and 
$\Delta(\sigma)$ be the corresponding boundary divisor of ${\FGcpt}$. 
Then the vanishing order $\nu_{\Delta(\sigma)}(\omega_{F})$ 
of $\omega_{F}$ along $\Delta(\sigma)$ is given by 
\begin{equation*}
\nu_{\Delta(\sigma)}(\omega_{F}) = 
\nu_{\sigma, geom}(F) - m = 
\begin{cases}
\nu_{\sigma}(F)-m & \sigma : \textrm{regular} \\ 
\frac{1}{2}\nu_{\sigma}(F)-m & \sigma : \textrm{irregular} 
\end{cases}
\end{equation*}
\end{proposition}

\begin{proof}
Let $\pi\colon (\mathcal{X}(I)')^{\Sigma_{I}} \to {\FGcpt}$ be the projection. 
By Propositions \ref{prop: geom interpret ord geom} and \ref{prop: K vs L}, 
we have 
\begin{equation*}
\nu_{D(\sigma)'}(\pi^{\ast}\omega_{F}) = \nu_{\sigma, geom}(F) - m. 
\end{equation*}
By Proposition \ref{prop: boundary ramification divisor} (1), 
$\pi$ is not ramified along $D(\sigma)'$, 
regardless of whether $\sigma$ is positive-definite or isotropic. 
This implies that 
$\nu_{D(\sigma)'}(\pi^{\ast}\omega_{F})=\nu_{\Delta(\sigma)}(\omega_{F})$. 
\end{proof}

When $\sigma=\sigma_{J}$ is isotropic, 
the above equality can be written as 
\begin{equation*}
\nu_{\Delta(\sigma_{J})}(\omega_{F}) = \nu_{J, geom}(F) - m, 
\end{equation*}
where $\Delta(\sigma_{J})$ is the boundary divisor of  ${\FGcpt}$ over $J$. 

By Gritsenko-Hulek-Sankaran \cite{GHS07}, 
every irreducible component of the ramification divisor of 
${\D}\to {\FG}$ has ramification index $2$ (and is defined by a reflection). 
Since every boundary divisor of ${\FGcpt}$ is of the form $\Delta(\sigma)$ 
for some ray $\sigma$ at some $0$-dimensional cusp $I$,  
Proposition \ref{prop: vanishing order relation} implies the following. 

\begin{corollary}\label{cor: extendability}
The $m$-canonical form $\omega_{F}$ extends holomorphically over the regular locus of ${\FGcpt}$ 
if and only if the following hold: 
\begin{enumerate}
\item $\nu_{R}(F)\geq m$ at every irreducible component $R$ of the ramification divisor of ${\D}\to {\FG}$. 
\item $\nu_{\sigma}(F)\geq m$ at every regular ray $\sigma$ for every $0$-dimensional cusp. 
\item $\nu_{\sigma}(F)\geq 2m$ at every irregular ray $\sigma$ for every irregular $0$-dimensional cusp. 
\end{enumerate}
\end{corollary}

Note that 
extendability at the boundary divisors over the $1$-dimensional cusps 
is encoded in the conditions (2), (3) at isotropic rays $\sigma$ for adjacent $0$-dimensional cusps.

\subsection{Low slope cusp form criterion}\label{ssec: cusp form criterion}

We now arrive at our principal purpose. 
Theorem \ref{thm: intro modified criterion} follows from the case $k<b$ in the following. 

\begin{theorem}\label{thm: low slope}
Let $L$ be a lattice of signature $(2, b)$ with $b\geq 9$. 
Let ${\G}$ be a subgroup of ${\OL}$ of finite index. 
We take a ${\G}$-admissible collection of fans $\Sigma=(\Sigma_{I})$ 
such that $\Sigma_{I}$ is basic with respect to ${\UIZZ}={\UIQ}\cap {\GG}$ at each $0$-dimensional cusp $I$. 
Assume that we have a cusp form $F$ of some weight $k$ and character with respect to ${\G}$ 
satisfying the following: 
\begin{enumerate}
\item At every irreducible component $R$ of the ramification divisor of ${\D}\to {\FG}$, 
we have $\nu_{R}(F)/k > 1/b$. 
\item At every regular ray $\sigma$ of $\Sigma_{I}$ at every $0$-dimensional cusp $I$, 
we have $\nu_{\sigma}(F)/k > 1/b$. 
\item At every irregular ray $\sigma$ of $\Sigma_{I}$ at every irregular $0$-dimensional cusp $I$, 
we have $\nu_{\sigma}(F)/k > 2/b$. 
\end{enumerate}
Then ${\FG}$ is of general type. 
\end{theorem}

\begin{proof}
The following argument is a slight modification of the proof of 
\cite{GHS07} Theorem 1.1, avoiding the use of a neat cover. 

Replacing $F$ with its power, which does not change the slopes $\nu_{\ast}(F)/k$, 
we may assume that the character $\chi$ is trivial. 
We first consider the case $b\nmid k$. 
By further replacing $F$ with its power $F^{2^{N}}$, 
where $N$ is determined by 
$[k/b]+2^{-N-1} \leq  k/b < [k/b]+2^{-N}$, 
we may assume that 
$k/b\geq [k/b]+1/2$ so that 
$[2k/b]=2[k/b]+1$. 
We write $N_{0}=[k/b]+1$. 
Then $F$ has vanishing order $\geq N_{0}$ at the ramification divisors of ${\D}\to {\FG}$ 
and at the regular boundary divisors, 
and  vanishing order $\geq 2N_{0}$ at the irregular boundary divisors. 
We denote by $M_{l}({\G})$ the space of ${\G}$-modular forms of weight $l$ with trivial character. 
For an even number $m$ we consider the subspace 
$V_{m}=F^{m}\cdot M_{(bN_{0}-k)m}({\G})$ 
of $M_{bN_{0}m}({\G})$. 
Modular forms in $V_{m}$ have vanishing order $\geq mN_{0}$ at the interior ramification divisors 
and at the regular boundary divisors, 
and vanishing order $\geq 2mN_{0}$ at the irregular boundary divisors. 
Thus the corresponding $mN_{0}$-canonical forms extend holomorphically over 
the regular locus of ${\FGcpt}$ by Corollary \ref{cor: extendability}. 
By our choice of $\Sigma$, 
${\FGcpt}$ has canonical singularities at the boundary points by Proposition \ref{thm: singularity}, 
and the interior ${\FG}$ has canonical singularities by 
Gritsenko-Hulek-Sankaran \cite{GHS07}. 
Therefore these $mN_{0}$-canonical forms extend holomorphically over 
a desingularization $X$ of ${\FGcpt}$. 
Since $bN_{0}>k$, we have 
\begin{equation*}
\dim V_{m} = \dim M_{(bN_{0}-k)m}({\G}) \sim c\cdot m^{b} \qquad (m\to \infty) 
\end{equation*}
for some $c>0$, 
so we find that $K_{X}$ is big. 

When $b\mid k$, we replace $F$ with the product of a sufficiently large power of $F$ 
and a modular form of weight indivisible by $b$. 
This perturbs the slopes $\nu_{\ast}(F)/k$ only by $\varepsilon$, 
so the inequalities in (1) -- (3) still hold.  
Then the same argument works. 
\end{proof}

\begin{remark}
If we replace "$>$" in the conditions (1) -- (3) by "$\geq $", 
then the conclusion will be weakened to  
"${\FG}$ has nonnegative Kodaira dimension". 
A power of $F$ gives a nonzero pluricanonical form. 
\end{remark}

Geometric explanation of Theorem \ref{thm: low slope} is as follows. 
We have the ${\Q}$-linear equivalence 
\begin{equation*}
K_{{\FGcpt}} \sim_{{\Q}} b\mathcal{L} - B/2 - \Delta_{reg} - \Delta_{irr} 
\end{equation*}
over ${\FGcpt}$, 
where $B$ is the interior branch divisor and 
$\Delta_{reg}, \Delta_{irr}$ are the regular and irregular boundary divisors respectively. 
The coefficients of $B$ and $\Delta_{irr}$ will be multiplied by $2$ 
when pulled back to local charts. 
The existence of the cusp form $F$ means that 
$b'\mathcal{L} - B/2 - \Delta_{reg} - \Delta_{irr}$ 
is ${\Q}$-effective for some $b'<b$, $b'\in {\Q}$. 
(To be explicit, $b'=k/N_{0}$ in the case $b\nmid k$ in the proof.) 
Thus we have 
\begin{equation*}
K_{{\FGcpt}} 
 \sim  
({\Q}\textrm{-effective}) + (b-b')\mathcal{L} 
 =  ({\Q}\textrm{-effective}) + (\textrm{big}) 
 =  (\textrm{big}), 
\end{equation*}
and the singularities do not impose obstruction.


\end{document}